\documentclass[reqno,12pt]{amsart}

\usepackage{amsmath}
\usepackage{amsfonts}
\usepackage{amssymb}
\usepackage{amsbsy}
\usepackage{amsthm}
\usepackage{mathrsfs}
\usepackage{latexsym}
\usepackage{caption,subcaption,color,booktabs}
\usepackage{graphicx}
\usepackage{chemfig,chemarrow}

\usepackage[switch*]{lineno}
\usepackage{fourier}
\usepackage{setspace}
\usepackage{lipsum}
\makeatletter
\g@addto@macro{\endabstract}{\@setabstract}
\newcommand{\authorfootnotes}{\renewcommand\thefootnote{\@fnsymbol\c@footnote}}%
\makeatother

\numberwithin{equation}{section}

\newtheorem{thm}{Theorem}[section]
\newtheorem{cor}[thm]{Corollary}
\newtheorem{lem}[thm]{Lemma}
\newtheorem{pro}[thm]{Proposition}

\newtheorem{ass}[thm]{Assumption}

\newcommand{\zb}{\boldsymbol z}
\newcommand{\ub}{\boldsymbol u}
\newcommand{\rb}{\boldsymbol r}
\newcommand{\cb}{\boldsymbol c}
\newcommand{\yb}{{\boldsymbol y}}
\newcommand{\eb}{\boldsymbol e}
\newcommand{\pb}{\boldsymbol p}
\newcommand{\xb}{\boldsymbol x}
\newcommand{\vb}{\boldsymbol v}
\newcommand{\wb}{\boldsymbol w}

\newcommand{\thetab}{\boldsymbol \theta}
\newcommand{\alphab}{\boldsymbol \alpha}
\newcommand{\betab}{\boldsymbol \beta}

\newcommand{\sigmab}{\boldsymbol \sigma}
\newcommand{\qb}{\boldsymbol q}
\newcommand{\mub}{\boldsymbol \mu}

\begin{document}

\title{}
\begin{center}
\LARGE Large deviations for two scale chemical kinetic processes \par \bigskip

\normalsize
\authorfootnotes
Tiejun Li \footnote{email: {\it tieli@pku.edu.cn}}
\and
Feng Lin\footnote{email: {\it math\_linfeng@pku.edu.cn}}
\par \bigskip

LMAM and School of Mathematical Sciences,  Peking University,\\
Beijing 100871, China \par
\end{center}

\date{}
\maketitle

\section*{Abstract}

We formulate the large deviations for a class of two scale chemical kinetic processes motivated from biological applications. The result is successfully applied to treat a genetic switching model with positive feedbacks. The corresponding Hamiltonian is convex with respect to the momentum variable as a by-product of the large deviation theory. This property ensures its superiority in the rare event simulations compared with the result obtained by formal WKB asymptotics. The result is of general interest to understand the large deviations for multiscale problems.

\section{Introduction}

We will investigate the large deviations for a class of two scale chemical kinetic processes with the slow variable $\zb_n \in \mathbb{N}^{d}/n$ satisfying
\begin{equation}\label{eq:CKS1}
 \zb_n(t)=\zb_n(0)+\sum_{i=1}^{S}\frac{1}{n}P_{i}\left(n\int_0^t\lambda_i(\zb_n(s),\xi_n(s))ds\right)\ub_i
 \end{equation}
subject to some fixed initial state $\zb_{n}(0)=\zb^0$, where $\{P_{i}(t)\}_{i=1,\ldots,S}$ are independent uni-rate Poisson processes, $\lambda_i\in \mathbb{R}^{+}$ is called the propensity function which characterizes the reaction rate of the $i$th reaction and $\ub_i\in \mathbb{Z}^{d}$ is called the state change vector.
The number $n\in \mathbb{N}$ corresponds to the system volume, thus $\zb_{n}$ has the meaning of concentration (number of molecules per volume) for the considered  kinetic system.  The fast variable $\xi_n\in \mathbb{Z}_{D}:=\{1,2,\cdots,D\}$ is a simple jump process with the time dependent rate $nq_{ij}(\zb_n(t))$ from state $i$ to $j$ at time $t$. With this mathematical setup, the processes $\zb_n(t)$ and $\xi_n(t)$ are fully coupled to each other and the infinitesimal generator $\mathcal {L}_n$ of this system has the form
\begin{align}\label{eq:fullsystem}
\mathcal {L}_nh(\zb,i)=&n\sum_{l=1}^S \lambda_l(\zb,i)[h(\zb+\ub_l/n,i)-h(\zb,i)]\nonumber\\
&+n\sum_{\substack{j=1 \\ j\ne i}} ^D  q_{ij}(z)[h(\zb,j)-h(\zb,i)],
\end{align}
where $\zb\in \mathbb{N}^d/n$ , $i\in \mathbb{Z}_{D}$ and $h$ is any compactly supported smooth function of $\zb$ for each $i$. For more about the notations and the backgrounds on the chemical kinetic processes, the readers may be referred to \cite{Kurtz1986,Gillespie1977}.

The above problem is motivated by our recent rare event study in the biological applications \cite{Assaf2011,LiLin2016,Lv2014}.  In a cell, the reactions underlying gene expression usually involve
low copy number of molecules, such as DNA, mRNAs and transcription factors, so the stochasticity in gene regulation process is inevitable even under constant environmental conditions \cite{Elowitz2002}. When the number of the molecules for all species goes to infinity and the law of mass action holds for the propensity functions, one gets the well-known large volume limit or Kurtz's limit, which gives the deterministic reaction rate equations for the concentration of the species \cite{Kurtz1972}. The convergence result can be further refined to the large deviation type \cite{Shwartz1995}. Recently, the following typical biological model with positive feedbacks is utilized to investigate the robustness of the genetic switching system \cite{Assaf2011,Chen2015,LiLin2016,Lv2014}.
$$
\begin{array}{ccccc}
{\rm DNA_{in}} &&\emptyset&&\emptyset\\
${\tiny G($Z_2$)}$  \upharpoonleft \downharpoonright ${\tiny $F(Z_2)$}$ &&\uparrow\gamma&&\uparrow 1\\
{\rm DNA_{act}} & \autorightarrow{${\tiny a}$}{} & ${{\rm mRNA} ($Z_1$)}$ & \autorightarrow{${\tiny\gamma b}$}{} & ${{\rm Protein} ($Z_2$)}$
\end{array}
$$
 This problem is a special case of our formulation shown at the beginning of this paper for $d=2,\ D=2$ and $S=4$. Denote $n$ the system size and $\zb=(z_1,z_2)=(Z_1,Z_2)/n$ the slow variables after taking large volume scaling, where $Z_1$ and $Z_2$ are the number of mRNA and protein molecules, respectively. Since there is only one molecule of DNA at active (${\rm DNA_{act}}$) or inactive state (${\rm DNA_{in}}$), for better use of notation, we take the fast variable $\xi\in \{0,1\}$ instead of $\{1,2\}$ to represent that the DNA is at inactive ($\xi=0$) or active state ($\xi=1$), respectively. By taking into account the scaling of parameters
$$a \sim nb^{-1},\quad \ F(Z_2),G(Z_2)\sim O(n)\quad \mbox{if}\quad Z_1,Z_2\sim O(n),$$
we further assume
\begin{equation}
F(Z_2)=n f(z_2), \quad G(Z_2)=n g(z_2).
\end{equation}
This assumption holds when we consider a Hill-function type jump rates with Hill coefficient 2 and large volume scaling for equilibrium constants \cite{Assaf2011}. Thus, we have the rescaled jump rates for DNA
\begin{equation}\label{rate of xi}
q_{01}(\zb)= f(z_2),\quad q_{10}(\zb)= g(z_2),
\end{equation}
and the following list of reactions associated with slow variables as shown in Table \ref{rates table}.
\begin{table}[h]\caption{Reaction schemes and parameters}\label{rates table}
\begin{tabular}{l | l | l }\toprule
Reaction scheme & Propensity function & State change vector\\ \midrule
$\text{DNA}_{\text{act}}\rightarrow \text{mRNA}$ & $\lambda_{1}(z_1,z_2,\xi) =b^{-1}\xi $ & $\ub_{1} = (1,0)$ \\
$\text{mRNA}\rightarrow \emptyset$ & $\lambda_{2}(z_1,z_2,\xi) =\gamma z_1$ & $\ub_{2} = (-1,0)$ \\
$\text{mRNA}\rightarrow \text{Protein}$ & $\lambda_{3}(z_1,z_2,\xi) =\gamma b z_1$ & $\ub_{3} = (0,1)$ \\
$\text{Protein}\rightarrow \emptyset$ & $\lambda_{4}(z_1,z_2,\xi) =z_2$ & $\ub_{4} = (0,-1)$\\ \midrule
\end{tabular}
\end{table}

The infinitesimal generator of this process has the form
\begin{align}\label{eq:Ln for model}
\mathcal {L}_nh(\zb,i)=& n\sum_{l=1}^4 \lambda_l(\zb,i)\Big(h(\zb+n^{-1}\ub_l,i)-h(\zb,i)\Big)\nonumber\\
&+n \bigg(f(z_2)[h(\zb,1)-h(\zb,0)]+g(z_2) [h(\zb,0)-h(\zb,1)]\bigg)
\end{align}
for $i=0,1$.
One can obtain a mean field ODE system as
\begin{equation}\label{eq: odes}
\frac{d z_1}{dt}=\frac{b^{-1} f(z_2)}{ f(z_2)+ g(z_2)}-\gamma z_1,\quad \frac{d z_2}{dt}=\gamma b z_1- z_2
\end{equation}
when $n$ goes to infinity through the perturbation analysis for the infinitesimal generator \cite{Kurtz1973,Lv2014,Papa1977}.
 With suitable choice of the functions $F(Z_2)$ and $G(Z_2)$, the final mean field ODEs have two stable stationary points and there are noise
 induced transitions between these two states when $n$ is finite.  To understand the robustness of the genetic switching, the biophysicists employed the WKB ansatz to the stationary distribution \cite{Assaf2011}
\begin{equation}\label{eq:WKB}
P(Z_1,Z_2)\sim \exp[-nS(z_1,z_2)]
\end{equation}
and obtained a steady state Hamilton-Jacobi equation $H(z_1,z_2,\nabla S) =0$. Mathematically the function $S$ resembles the role of the quasi-potential of the stochastic dynamical system \cite{Freidlin1998, Lv2015, Zhou2016} but it is not sure whether it is the case in the current stage. Another related physics approach to study a similar switching system is to utilize the spin-boson path integral formalism in quantum field theory and then take the semiclassical approximation and  adiabatic limit \cite{LiLin2016,Wang2013}. Both approaches are difficult to be rationalized in mathematical sense. So how to formulate this problem in a mathematically rigorous way?  To resolve this issue, we have to answer the following two fundamental questions.
\begin{enumerate}
\item Question 1. What is the large deviation principle (LDP) associated with the system \eqref{eq:fullsystem}? Presumably, we can obtain the Lagrangian from the large deviation analysis, then get the Hamiltonian $H$ through the Legendre-Fenchel transform.
\item Question 2. What is the relation between the rigorously obtained Hamiltonian $H$ in the above question and the Hamiltonian obtained via WKB asymptotics?
\end{enumerate}

The aim of this paper is to make an exploration on these two questions. To do this, we first note that the large volume limit no longer holds in the current example. Although the mRNA and protein copy numbers scale as $V$, we have only one DNA, which switches between the active and inactive states. This fact excludes the direct applicability of the LDP results in \cite{Shwartz1995}.  However, the fast switching between the two states of the DNA ensures the averaging technique still valid as shown in \eqref{eq: odes} by taking the quasi-equilibrium limit \cite{E2007,Kurtz2013,Lv2014}. We will show that  the LDP analysis is also feasible by incorporating the Donsker-Varadhan type large deviations. Indeed, similar situation has been nicely discussed by Liptser \cite{Liptser1996} and Veretennikov \cite{Veretennikov2000, Veretennikov1999} for two-scale diffusions like
\begin{eqnarray}
 &&dX_n(t)=A(X_n(t),\xi_n(t))dt+\frac{1}{\sqrt{n}}B(X_n(t),\xi_n(t))dW_t,\label{eq: lipste1}\\
 &&d\xi_n(t)=nb(\xi_n(t))dt+\sqrt{n}\sigma(\xi_n(t))dV_t.\label{eq:lipster2}
\end{eqnarray}
The main idea of this paper is to generalize the result in \cite{Liptser1996} to our two-scale chemical kinetic processes. As we will see, although the framework is similar, we have to deal with the
technicalities brought by the jump processes and the full coupling between the fast and slow variables ($\xi_n$ is independent of $X_n$ in \eqref{eq:lipster2}).

To state the main results of this paper, let us introduce the occupation measure $\nu_n$ on $([0,T]\times \mathbb{Z}_{D}, \mathcal{B}([0,T])\otimes \mathcal{B}(\mathbb{Z}_{D})$ corresponding to $\xi_n$
\begin{equation}\label{eq:OccMeasDef}
\nu_n(\Delta \times \Gamma)=\int_0^T{\bf 1} (t\in \Delta,\xi_n \in \Gamma)dt ,\ \ \Delta \in \mathscr{B}([0,T]), \Gamma \in \mathscr{B}(\mathbb{Z}_{D}),
\end{equation}
where $T$ is any fixed positive real number. Denote $\mathbb{D}^d[0,T]$ the space of $d$-dimensional vector functions on $[0,T]$ whose components are right continuous with left-hand limits, $\mathbb{M}_L[0,T]$ of finite measures $\nu = \nu(dt,i)$ on $([0,T]\times \mathbb{Z}_{D}, \mathscr{B}([0,T])\otimes \mathscr{B}(\mathbb{Z}_{D}))$ which are absolutely continuous with respect to $dt$ and have Lebesgue time marginals, i.e. we have $\nu(dt, i)= n_{\nu}(t,i)dt$, $n_{\nu}(t,i)\ge 0$ and $\sum_{i=1}^Dn_{\nu}(t,i)=1$. The $\nu_{n}$ we considered always belongs to $\mathbb{M}_L[0,T]$. Take the metric $\rho^{(2)}$  on $\mathbb{M}_L[0,T]$ as the L\'evy-Prohorov metric and $\rho^{(1)}$  on $\mathbb{D}^{d}[0,T]$ as the Skorohod metric defined as
\begin{equation}
\rho^{(1)}(\rb,\tilde\rb) = \inf_{\lambda\in \mathcal{F}}\Big\{\|\lambda\|^{\circ} \vee \sup_{t\in[0,T]}\|\rb(t)-\tilde\rb(\lambda(t))\|\Big\},
\end{equation}
where $\|\cdot\|$ is the Euclidean norm in the corresponding space, $\mathcal{F}$ is the collection of strictly increasing functions $\lambda(t)$ such that $\lambda(0)=0$ and $\lambda(T)=T$, and
$$
\|\lambda\|^\circ := \sup_{0\le s<t\le T}\left|\log \frac{\lambda(t)-\lambda(s)}{t-s}\right|.
$$
$\mathbb{D}^d[0,T]$ and $\mathbb{M}_L[0,T]$ are complete and separable spaces with $\rho^{(1)}$ and $\rho^{(2)}$, respectively \cite{BillingsleyBook}. Our task is to establish the LDP for the pair $(\zb_n,\nu_n)$ in metric space $(\mathbb D^d[0,T]\times\mathbb{M}_L[0,T],\rho^{(1)}\times\rho^{(2)})$. 

This paper is organized as follows. In Section 2, we present the main large deviation theorem and give the rate functional of the whole system. By using the contraction principle and the Legendre-Fenchel transform we get the Hamiltonian related to the slow variable $\zb_n$.
As a concrete application, we then study the genetic switching model and compare the difference between the rigorously obtained Hamiltonian and that obtained by WKB ansatz.
In Sections 3 and 4, we give the proof of the main theorem. Due to the technicalities to handle the non-negativity constraint for $\rb$,  we decompose the proof procedure into two steps. In Section 3, we prove the LDT theorem by relaxing the bounded domain condition to the whole space case. The upper bound estimate is standard in some sense. However, the proof of the lower bound is technical because of the full coupling between the fast and slow variables. The resolution is based on the approximation and change-of-measure approach. The central idea is to make a piecewise linear approximation to any given path and occupation measure $(\rb,\nu)$ by $(\yb,\pi)$ at first, and then construct suitable new processes $(\bar \zb_n,\bar \nu_n)$ such that $\mathbb{P}-\lim_{n \to \infty}\rho^{(1)}(\bar \zb_n,\yb)=0$ and $\mathbb{P}-\lim_{n \to \infty}\rho^{(2)}(\bar \nu_n,\pi)=0$. This turns out to be technical and one key part of the whole paper. In Section 4, we strengthen the result to the half space case. Some details are left in the Appendix.

This paper should be considered as the companion  of \cite{LiLin2016, Lv2014} for studying the rare events in genetic switching system, and it is of general interest to understand the large deviations for multiscale problems \cite{E2005,E2007}.

\section{Main result and its application}\setcounter{equation}{0}

\subsection{Main theorem}
We need the following technical assumptions for our main result.
\begin{ass}\label{ass:second} Let $W := \overline {\mathbb R_{+}^d}$. Assume the following regularity conditions for the propensity functions and jump rates hold.
\begin{enumerate}
\item  (a)~For each $i\in\{1,2,\ldots,S\}, j\in \mathbb{Z}_{D}$ and all $\zb, \xb\in W$, the Lipschitz condition holds
\begin{equation}\label{eq:LambdaLip}
|\lambda_i(\zb,j)-\lambda_i(\xb,j)|\le L \|\zb-\xb\|.
\end{equation}

\noindent(b)~For each $i\in\{1,2,\ldots,S\}, j\in \mathbb{Z}_{D}$ and all $\zb \in W^\circ$, $\lambda_i(\zb,j)>0$.

\noindent(c)~For each $\xb \in \partial W$ and $\yb \in \mathcal{C}\{\ub_j| \lambda_j(\xb)>0\}$, we have $\xb+s\yb \in W$ for some $s \in (0,\infty)$ , where $\mathcal{C}\{\ub_j\}$ is the positive cone spanned by the vectors $\{\ub_j\}$ defined as
\begin{equation}
\mathcal{C}\{\ub_j\} := \big\{\vb| \text{there exist}~\alpha_j \ge 0~\text{such that}~\vb= {\sum}_j \alpha_j \ub_j\big\}.
\end{equation}

\item  For each $i,j\in \mathbb{Z}_{D}$, $\log q_{ij}(\zb)$ are bounded and Lipschitz continuous with respect to $\zb\in W $.

\end{enumerate}
\end{ass}
These assumptions hold in our application example in Section \ref{sec:application}.

\begin{thm}\label{second result}
Under the Assumption \ref{ass:second}, the family $(\zb_n,\nu_n)$ defined by \eqref{eq:CKS1} and \eqref{eq:OccMeasDef}
obeys the LDP in $(\mathbb{D}^d[0,T]\times\mathbb{M}_L[0,T],\rho^{(1)}\times\rho^{(2)})$ with a good rate functional $I(\rb,\nu) = I_{s}(\rb,\nu)+I_{f}(\rb,\nu)$, i.e.
\begin{itemize}
\item[(0)]$I(\rb,\nu)$ values in $[0,+\infty]$ and  its level sets are compact in $(\mathbb{D}^d[0,T]\times \mathbb{M}_L[0,T],\rho^{(1)}\times \rho^{(2)})$,
\item[(1)]for every closed set $F \in \mathbb{D}^d[0,T]\times\mathbb{M}_L[0,T]$,
\begin{equation}\label{eq:UpperBound}
\mathop{\lim\sup}_{n\to \infty}\frac{1}{n}\log\mathbb{P}((\zb_n,\nu_n)\in F)\le -\mathop{\inf}_{(\rb,\nu)\in F}I(\rb,\nu),
\end{equation}
\item[(2)]for every open set $G \in \mathbb{D}^d[0,T]\times\mathbb{M}_L[0,T]$,
\begin{equation}\label{eq:LowerBound}
\mathop{\lim\inf}_{n\to \infty}\frac{1}{n}\log\mathbb{P}((\zb_n,\nu_n)\in G)\ge -\mathop{\inf}_{(\rb,\nu)\in G}I(\rb,\nu),
\end{equation}
\end{itemize}
where the rate functional for the slow variables
\begin{equation}
I_{s}(\rb,\nu) = \left\{
\begin{array}{cl}
   \int_0^T L_{s}(\rb(t),\dot{\rb}(t),n_\nu(t,\cdot))dt, & d\rb(t)=\dot{\rb}(t)dt,   \\
   \infty,  & \text{otherwise},  \\
 \end{array}
 \right.
\end{equation}
 \begin{equation}\label{eq:LsDef}
 L_{s}(\zb,\betab,\wb) =  \mathop {\sup }\limits_{\pb \in \mathbb{R}^d}\left(\left\langle\pb,\betab\right\rangle-H_{s}(\zb,\pb,\wb)\right),
 \end{equation}
 \begin{equation}\label{eq:Hs}
H_{s}({\zb,\pb,\wb})  = \sum\limits_{i = 1}^S \sum\limits_{j = 1}^D{{\lambda _i}(\zb,j)w_j\left({e^{\left\langle {\pb,\ub_i} \right\rangle }} - 1\right)},
 \end{equation}
 and the rate functional for the fast variables
 \begin{equation}
 I_{f}(\rb,\nu)=\int_0^T S(\rb(t),n_\nu(t,\cdot))dt,
 \end{equation}
\begin{equation}\label{eq:SDef}
 S(\zb,\wb)=\mathop{\sup}\limits_{\sigmab \in \mathbb{R}^D}S(\zb,\wb,\sigmab),
\end{equation}
\begin{equation}S(\zb,\wb,\sigmab)=
-\sum_{i,j=1}^{D}w_iq_{ij}(\zb)\left(e^{\left\langle\sigmab,\eb_{ij}\right\rangle}
-1\right).
\end{equation}
\noindent Here we take the notation $\nu(dt,\cdot)=n_{\nu}(t,\cdot)dt$, thus $n_{\nu}(t,\cdot)$ is a probabilistic vector $\left(n_{\nu}(t,1),n_{\nu}(t,2),\ldots,n_{\nu}(t,D)\right)$. $\wb=(w_1,w_2,$ $\ldots,w_D)$,  and $\langle\cdot,\cdot\rangle$ is the inner product in the Euclidean space. $\eb_{ij}=\eb_{i}-\eb_{j}$ and  $\{\eb_{i}\}_{i=1}^{D}$ are canonical basis in Euclidean space $\mathbb{R}^{D}$.  We take the convention that $\rb(t)$ is absolutely continuous with respect to time when we use the notation $d\rb(t)=\dot{\rb}(t)dt$, and $S$ is a  function of $(\zb,\wb)$  (or $(\zb,\wb,\sigmab)$) when we use $S(\zb,\wb)$ (or $S(\zb,\wb,\sigmab)$) by default.
\end{thm}

The proof of  Theorem \ref{second result} relies on first establishing a weaker statement based on the following stronger assumption on the whole space.
\begin{ass}\label{ass:main} Regularity for the propensity functions and jump rates.
\begin{enumerate}
\item  For each $i\in\{1,2,\ldots,S\}, j\in \mathbb{Z}_{D}$, $\log \lambda_i(\zb,j)$ is bounded and Lipschitz continuous with respect to $\zb\in \mathbb{R}^{d}$.
\item  For each $i,j\in \mathbb{Z}_{D}$, $\log q_{ij}(\zb)$ are bounded and Lipschitz continuous with respect to $\zb\in \mathbb{R}^{d}$.
\end{enumerate}
\end{ass}
\noindent This covers Assumption \ref{ass:second}. Mathematically we express (1) as
\begin{align}\label{eq:LambdaLUBound}
\frac{1}{\Lambda} \le \lambda_i(\zb,j) \le\Lambda, \quad\Lambda >1
\end{align}
for any $\zb\in \mathbb{R}^d$, $i\in\{1,2,\ldots,S\}$ and $ j\in \mathbb{Z}_{D}$. And in this stronger set-up we simply denote the positive cone generated by $\{\ub_j\}$ as
\begin{equation}\label{eq:ConeDef}
\mathcal{C} := \big\{\vb| \text{there exist}~\alpha_j \ge 0~\text{such that}~\vb= {\sum}_j \alpha_j \ub_j\big\}.
\end{equation}

\begin{thm}\label{Main result}
The large deviation result in Theorem \ref{second result} holds for $(\zb_{n},\nu_n) \in \mathbb{D}^d[0,T]\times \mathbb{M}_L[0,T]$ under the Assumption \ref{ass:main}.
\end{thm}

As a straightforward application of the contraction principle, we have
\begin{cor} The slow variables $\zb_n$ obeys the LDP in $(\mathbb{D}^d[0,T], \rho^{(1)})$ with the rate functional
\begin{equation}
I(\rb)=\mathop{\inf}_{\nu\in \mathbb{M}_L[0,T]}(I_{s}(\rb,\nu)+I_{f}(\rb,\nu)).
\end{equation}
\end{cor}

Define the set of probabilistic transition kernels on $\mathbb{Z}_D$ as
$\Delta_{D}=\{\wb: w_1, w_2,\cdots,w_D \ge 0, \sum_{i=1}^D w_i=1\}$
where $\wb=(w_1, w_2,\cdots,w_D)$. We also define the reduced Lagrangian as
\begin{equation}
L(\zb,\betab) = \mathop{\inf}_{\wb \in \Delta_{D}}\left\{L_{s}(\zb,\betab,\wb)+S(\zb,\wb)\right\}.
\end{equation}
For convenience, we will abuse the notation $n_{\nu}\in \mathbb{M}_L[0,T]$ and $\nu \in \mathbb{M}_L[0,T]$ in later texts.

\begin{lem}\label{comm}
For any $\rb(\cdot)$ which is absolutely continuous, we have
\begin{align}\notag
&\mathop{\inf}_{n_\nu \in \mathbb{M}_L[0,T]}\int_0^T L_{s}(\rb(t),\dot{\rb}(t),n_\nu(t,\cdot))+S(\rb(t),n_\nu(t,\cdot))dt\\\label{commutable}
&=\int_0^T L(\rb(t),\dot \rb(t))dt.
\end{align}
\end{lem}
\begin{proof} First let us show the measurability of the integrand on the right hand side of Eq. \eqref{commutable}.
By Lemma \ref{convex}, $L_s(\zb,\betab,\wb)+S(\zb,\wb)$ is convex in $\wb$. So $L_s(\zb,\betab,\wb)+S(\zb,\wb)$ is continuous with respect to $\wb$ in the set $\Delta_D^\circ \subset \mathbb{R}^D$ and the interior of the low dimensional boundaries of $\Delta_D$.
Choosing a countable dense subset $\big\{\wb^k\big\}_{k=1}^\infty$ in $\Delta_D$, we have
\begin{align}\label{wk}
L(\rb(t),\dot \rb(t))=\mathop{\inf}_{k\ge 1}\left\{L_{s}(\rb(t),\dot \rb(t),\wb^k)+S(\rb(t),\wb^k)\right\}
\end{align}
for every $\rb$ by the continuity condition. The measurability is a standard result with this formulation.

It is straightforward to have that
\begin{align*}
&\mathop{\inf}_{n_\nu\in \mathbb{M}_L[0,T]}\int_0^T L_{s}(\rb(t),\dot{\rb}(t),n_\nu(t,\cdot))+S(\rb(t),n_\nu(t,\cdot))dt\\
&\ge\int_0^T L(\rb(t),\dot \rb(t))dt.
\end{align*}
Now let us show the converse part. For any given $\epsilon>0$, define the sets
\begin{align*}
A_k=\bigg\{&t\in[0,T]: L(\rb(t),\dot \rb(t))-\\
&\Big(L_{s}(\rb(t),\dot \rb(t),\wb^k)+S(\rb(t),\wb^k)\Big)\ge-\epsilon/T\bigg\}
\end{align*}
for $k\ge 1$.  We have that $A_k$ are measurable sets since $L_{s}(\rb(t),\dot \rb(t),\wb^k)+S(\rb(t),\wb^k)$ and $L(\rb(t),\dot \rb(t))$ are both measurable functions of $t$. Define the measurable functions
$$F_k(t) = \left\{
\begin{array}{cl}
   k, & t\in A_k,  \\
   +\infty,  & \text{otherwise} \\
 \end{array}
 \right.$$
for every $k\ge 1$ and
\begin{equation}
J(t)=\mathop{\inf}_{k\ge 1} F_k(t).
\end{equation}
It is not difficult to find that $J(t)<+\infty$ for any $t$,  $J(t)$ is measurable and takes values in positive integers.
With these definitions we have
\begin{align}\label{eq:lm-approx}
L(\rb(t),\dot \rb(t))\ge L_{s}(\rb(t),\dot \rb(t),\wb^{J(t)})+S(\rb(t),\wb^{J(t)})-\epsilon/T.
\end{align}
With $\wb^{J(t)}:=\big\{w^{J(t)}_1,w^{J(t)}_2,\ldots,w^{J(t)}_D\big\}$, define the occupation measure $\widehat\nu$
$$ \widehat\nu (dt,i)= w^{J(t)}_idt, \quad i\in \{1,2, \cdots, D\} . $$
Then $\widehat\nu \in \mathbb{M}_L[0,T]$, $n_{\widehat\nu}(t,i)= w^{J(t)}_i$ and
\begin{align*}
&\int_0^T L(\rb(t),\dot \rb(t))dt\\
\ge& \int_0^T L_{s}(\rb(t),\dot \rb(t),n_{\widehat\nu}(t,\cdot))+S(\rb(t),n_{\widehat\nu}(t,\cdot))dt-\epsilon\\
\ge& \mathop{\inf}_{n_\nu\in \mathbb{M}_L[0,T]}\int_0^T L_{s}(\rb(t),\dot{\rb}(t),n_\nu(t,\cdot))+S(\rb(t),n_\nu(t,\cdot))dt-\epsilon
\end{align*}
The proof is completed.
\end{proof}

By Lemma \ref{comm}, we have
\begin{align}
 I(\rb)&=\mathop{\inf}_{\nu\in \mathbb{M}_L[0,T]}(I_{s}(\rb,\nu)+I_{f}(\rb,\nu))\nonumber\\
&=\mathop{\inf}_{n_\nu\in \mathbb{M}_L[0,T]}\int_0^T L_{s}(\rb(t),\dot{\rb}(t),n_\nu(t,\cdot))+S(\rb(t),n_\nu(t,\cdot))dt\nonumber\\
 &=\int_0^T L(\rb(t),\dot \rb(t))dt.
\end{align}

\begin{lem}
$L(\rb,\betab)$ is convex in $\betab$ .
\end{lem}
\begin{proof}
By Lemma \ref{Min-Max},
\begin{align*}
 L(\zb,\betab)&= \mathop{\inf}_{\wb \in \Delta_{D}}\left\{L_{s}(\zb,\betab,\wb)+S(\zb,\wb)\right\}\\
 &=\mathop{\inf}_{\wb \in \Delta_{D}}\mathop{\sup}_{\pb\in \mathbb{R}^d}(\langle\pb,\betab\rangle-H_{s}(\zb,\pb,\wb)+S(\zb,\wb))\\
 &=\mathop{\sup}_{\pb\in \mathbb{R}^d}\mathop{\inf}_{\wb \in \Delta_{D}}(\langle\pb,\betab\rangle-H_{s}(\zb,\pb,\wb)+S(\zb,\wb)).
\end{align*}
It is easy to see that $\mathop{\inf}_{\wb \in \Delta_{D}}(\langle\pb,\betab\rangle-H_{s}(\zb,\pb,\wb)+S(\zb,\wb))$ is linear in $\betab$, thus $L(\rb,\betab)$ is convex in $\betab$ according to Lemma \ref{convex}.
\end{proof}

It is well-known that the Lagrangian  $L_{s}$ does not have a closed form for the standard chemical reaction kinetic system, instead it is more convenient to investigate its dual Hamiltonian $H_{s}$ by  Legendre-Fenchel transform. The explicit form of the Hamiltonian is important for the numerics to study the rare events in systems biology \cite{Eric2008}. With similar idea, we have
\begin{align}\notag
H(\zb,\pb)&=\mathop{\sup}_{\betab\in \mathbb{R}^d}(\langle \pb,\betab\rangle-L(\zb,\betab))\\ \notag
&=\mathop{\sup}_{\betab\in \mathbb{R}^d}\left(\langle \pb,\betab\rangle-\mathop{\inf}_{\wb \in \Delta_{D}}\left\{L_{s}(\zb,\betab,\wb)+S(\zb,\wb)\right\}\right)\\ \notag
&=\mathop{\sup}_{\betab\in \mathbb{R}^d}\mathop{\sup}_{\wb \in \Delta_{D}}\left(\langle \pb,\betab\rangle-L_{s}(\zb,\betab,\wb)-S(\zb,\wb)\right)\\ \notag
&=\mathop{\sup}_{\wb \in \Delta_{D}}\mathop{\sup}_{\betab\in \mathbb{R}^d}\left(\langle \pb,\betab\rangle-L_{s}(\zb,\betab,\wb)-S(\zb,\wb)\right)\\
&=\mathop{\sup}_{\wb \in \Delta_{D}}(H_{s}(\zb,\pb,\wb)-S(\zb,\wb)). \label{eq:Hamiltonian}
\end{align}
A consequence about $H$ from its definition is that $H$ is convex with respect to $\pb$ from the convexity of $L$ and the Legendre-Fenchel transform \cite{Ellis1985}. Furthermore if the matrix $Q=(q_{ij})_{D\times D}$ is symmetrizable, $S(\zb,\wb)$ has an explicit expression
\cite{Chen1990}
\begin{equation}\label{Chen's}
S(\zb,\wb)=\frac{1}{2}\sum_i\sum_{j\ne i}\left[\sqrt{w_i q_{ij}(\zb)}-\sqrt{w_j q_{ji}(\zb)}\right]^2.
\end{equation}

\subsection{Application to the genetic switching model}\label{sec:application}

The formula \eqref{eq:Hamiltonian} has a nice application in the genetic switching model introduced before. In this model, we have $d=2,D=2$ and $S=4$. By parameters shown in \eqref{rate of xi} and
Table \ref{rates table}, we have
\begin{equation}
H_{s}(\zb,\pb,\wb)=b^{-1}w_1(e^{p_1}-1)+A(z_1,z_2,p_1,p_2),
\end{equation}
where $\zb=(z_1,z_2), \pb=(p_1,p_2), \wb=(w_0,w_1)$ (here we utilize the notation $\wb=(w_0,w_1)$ instead of $\wb=(w_1,w_2)$ as mentioned in the introduction since there is only one molecule of DNA) and
$A(z_1,z_2,p_1,p_2)=\gamma z_1(e^{-p_1}-1)+\gamma bz_1(e^{p_2}-1)+z_2(e^{-p_2}-1)$.
We  also have
\begin{eqnarray*}
S(\zb,\wb)=\left(\sqrt{w_0 f(z_2)}-\sqrt{w_1 g(z_2)}\right)^2.
\end{eqnarray*}
Applying \eqref{eq:Hamiltonian} with the constraints $w_0+w_1= 1$ and $w_0,w_1\ge 0$, we obtain the final Hamiltonian
\begin{equation}\label{eq: final H}
H(\zb,\pb)=b^{-1}s(e^{p_1}-1)-\left(\sqrt{(1-s)f(z_2)}-\sqrt{s g(z_2)}\right)^2+A(z_1,z_2,p_1,p_2),
\end{equation}
where $$s=\frac{1}{2}+\frac{s_1}{2\sqrt{s_1^2+4}}, \quad s_1=\frac{b^{-1}(e^{p_1}-1)+f(z_2)-g(z_2)}{\sqrt{f(z_2)g(z_2)}}.$$

It is instructive to compare this Hamiltonian with that obtained via WKB asymptotics. In \cite{Assaf2011}, another form of the Hamiltonian for this system is given via WKB asymptotics:
\begin{equation}\label{eq:HamiltonianWKB}
\tilde{H}(\zb,\pb)=A+g(z_2)^{-1}[A+b^{-1}(e^{p_1}-1)][f(z_2)-A],
\end{equation}
where $A=A(z_1,z_2,p_1,p_2)$. The relation between the Hamiltonian $\tilde{H}$ and $H$ is not clear so far. But one crucial difference is that $H$ is convex with respect to the momentum variable $\pb$ from the form  \eqref{eq:Hamiltonian}, while $\tilde H$ is not. It turns out this property is crucial for the numerical computations, especially for computing the transition path in geometric minimum action method (gMAM) \cite{Eric2008}. It is also interesting to observe that the quasi-potential $S(z_1,z_2)$ obtained from
$$H(\zb,\nabla S)=0\quad \text{ or }\quad \tilde H(\zb,\nabla S)=0$$
is the same even $H$ and $\tilde H$ are so different \cite{Lv2014}. It can be also verified that $H$ is not the convex hull of $\tilde{H}$ with respect to $\pb$. From the Hamilton-Jacobin theory, one may speculate that these two Hamiltonians are connected through some canonical transformation. But  it is only a plausible answer which is difficult to be verified even for this concrete example.

As the large deviation results give the sharpest characterization of the considered two-scale chemical kinetic system, we can obtain the deterministic mean field ODEs and the chemical Langevin approximation for the system based on the large deviations \cite{Dembo1993}, which corresponds to the law of large numbers (LLN) and the central limit theorem (CLT) for the process. Taking advantage of \eqref{eq: final H}, we get
\begin{equation}
\frac{\partial H}{\partial p_1} \Big|_{\pb = \boldsymbol{0}} = \frac{b^{ - 1}f(z_2)}
{f(z_2) + g(z_2)} - \gamma z_1,\qquad
\frac{\partial  H}{\partial p_2} \Big|_{\pb = \boldsymbol{0}} = \gamma bz_1 - z_2.
\end{equation}
The mean field ODEs defined by
\begin{equation}
\frac{dz_1}{dt}=\frac{\partial H}{\partial p_1} \Big|_{\pb = \boldsymbol{0}}\quad \mbox{and}\quad
\frac{dz_2}{dt}=\frac{\partial  H}{\partial p_2} \Big|_{\pb = \boldsymbol{0}}
\end{equation}
are exactly \eqref{eq: odes}.

Furthermore, we have
\begin{eqnarray*}
&&\frac{\partial ^2 H}{\partial p_1^2} \Big|_{\pb = \boldsymbol{0}} = \frac{{{b^{ - 1}}f(z_2)}}
{{f(z_2) + g(z_2)}} + \frac{{2{b^{ - 2}}f(z_2)g(z_2)}}
{{{{(f(z_2) + g(z_2))}^3}}} + \gamma z_1, \\
&&\frac{\partial ^2 H}{\partial p_2^2} \Big|_{\pb = \boldsymbol{0}} = \gamma bz_1 + z_2.
\end{eqnarray*}
This naturally leads to the following chemical Langevin approximation
 \begin{align}
\frac{dz_1}{dt}&=\left[\frac{b^{ - 1}f}
{f + g} - \gamma z_1\right]dt+\frac{1}{\sqrt{n}}\left[\sqrt{\frac{b^{-1}f}
{f + g} + \frac{2b^{-2}fg}{(f + g)^3}}dB^{1}_{t}-\sqrt{ \gamma z_1}dB^2_t\right],\notag\\
\frac{dz_2}{dt}&=\left[\gamma bz_1 - z_2\right]dt+\frac{1}{\sqrt{n}}\left[\sqrt{\gamma bz_1}dB^3_t-\sqrt{z_2}dB^4_t\right],\label{eq:langevin}
\end{align}
where $f, g$ are abbreviations of functions $f(z_2)$ and $g(z_2)$, and  $B^i_t$ $(i=1,$\ldots$,4)$ are independent standard Brownian motions. It is instructive to compare \eqref{eq:langevin} with a granted formulation by naively transplanting the Langevin approximation from the simple large volume limit \cite{Gillespie2000}, where the equation for $z_1$ reads
 \begin{align}
\frac{dz_1}{dt}&=\left[\frac{b^{ - 1}f}
{f + g} - \gamma z_1\right]dt+\frac{1}{\sqrt{n}}\left[\sqrt{\frac{b^{-1}f}
{f + g}}dB^{1}_{t}-\sqrt{ \gamma z_1}dB^2_t\right], \label{eq: usual langevin}
\end{align}
and the equation for $z_2$ is the same. It is remarkable that the Eq. \eqref{eq:langevin} has an additional term related to the noise $dB^{1}_{t}$. This additional fluctuation is induced by the fast switching of DNA states. Similar situation will also occur when we derive the chemical Langevin equations for enzymatic reactions, whereas  we should take the fluctuation effect of the fast switching into consideration if the considered scaling is in our regime. However, this point does not seem to be paid much attention in previous research. Similar situation is further discussed in \cite{LiLin2016}.

\subsection{A useful property of the Hamiltonian $H$}

The Hamiltonian $H(\zb,\pb)$ has some nice properties which can be utilized to simplify the computations in many cases. Assuming that $Q=(q_{ij})_{D\times D}$ is symmetrizable,
 according to \eqref{eq:Hamiltonian}, we have
$$ H(\zb,\pb)=\mathop{\sup}_{\wb \in \Delta_{D}} h(\zb,\pb,\wb),$$
where \begin{align*}
h(\zb,\pb,\wb)=& H_{s}(\zb,\pb,\wb) - S(\zb,\wb)\\
=&\sum\limits_{i = 1}^d  \sum\limits_{j = 1}^D{\lambda _i(\zb,j)w_j({e^{\left\langle {\pb ,{\ub_i}} \right\rangle }} - 1)}\\
&-\frac{1}{2}\sum_i\sum_{j\ne i}\left[\sqrt{w_i q_{ij}(\zb)}-\sqrt{w_j q_{ji}(\zb)}\right]^2.
\end{align*}
We will show that the supremum of $h$ in $\Delta_D$ can be only taken in the interior $\Delta_D^{\circ}$ of $\Delta_D$. To do this, we first note that $h$ is continuous in $\Delta_D$ and differentiable in $\Delta_D^{\circ}$. For any $\wb_b \in \partial(\Delta_D) $, define $\vb =\cb_0-\wb_b$ where $\cb_{0}=(1,1,\ldots,1)/D$ is the center of $\Delta_{D}$. It is easy to check that
\begin{equation}
\mathop{\lim}\limits_{t\to 0^+}\frac1t\Big(h(\zb,\pb,\wb_b+t\vb)-h(\zb,\pb,\wb_b)\Big) =  +\infty.
\end{equation}
This means that the supremum of $h$ can not be taken in $\partial(\Delta_D)$. Furthermore, since $h$ is strictly concave  in $\wb$, there exists only one point $\wb^*$ in $\Delta_D^{\circ}$, such that
$$\wb^*=\mathop{\arg\sup}_{\wb \in \Delta_D} h(\zb,\pb,\wb).$$
An important consequence of this fact is that we can get the derivative
\begin{align*}
\frac{\partial H(\zb,\pb)}{\partial \pb}&=\frac{d h(\zb,\pb,\wb^*(\pb))}{d\pb}\\
&=\frac{\partial  h(\zb,\pb,\wb^*)}{\partial \pb}+\frac{\partial  h}{\partial \wb}\Big |_{\wb=\wb^*}\frac{\partial \wb^*(\pb)}{\partial \pb}\\
&=\frac{\partial  H_{s}(\zb,\pb,\wb^*)}{\partial \pb}.
\end{align*}
This is very useful to simplify the derivations when utilizing the gMAM algorithm \cite{Eric2008} to explore the transition paths.

\section{Proof of Theorem \ref{Main result}}

We will mainly follow the framework in \cite{Liptser1996,Shwartz1995} to make the proof.  First we prove the upper bound and then the lower bound.
\subsection{Upper Bound}\setcounter{equation}{0}
The proof of upper bound \eqref{eq:UpperBound} is standard in some sense. It is difficult to estimate the probability of $(\zb_n,\nu_n) \in F$ directly. We proceed with the following steps.
Firstly, we approximate $\zb_n$ by $\tilde{\zb}_n$, where $\tilde{\zb}_n$ is an absolutely continuous path.
Secondly, for a given compact set, we can get an upper bound for $(\tilde{\zb}_n,\nu_n)$. Thirdly, we prove that after excluding a set of exponentially small probability, $\tilde{\zb}_n$ and $\nu_n$ stay in compact sets, which means that $\tilde{\zb}_n$ and $\nu_n$ are exponentially tight sequence. And finally, we get the desired result by combing the previous steps with further estimates.

 Before proceeding to the proof, let us denote $\mathbb C^d[0,T]$ the collection of all continuous functions of $t\in[0,T]$ with values in $\mathbb R^d$. Define the  sup-norm for any $\rb,\tilde\rb \in \mathbb C^d[0,T]$
$$\rho_c^{(1)}(\rb, \tilde \rb):=\mathop{\sup}\limits_{0\le t \le T}\|\rb(t)-\tilde \rb(t)\|.$$
We have that $(\mathbb C^d[0,T],\rho_c^{(1)})$ is a Polish space.  The metric $\rho_c^{(1)}$ is stronger than $\rho^{(1)}$ on $\mathbb D^d[0,T]$. As a  consequence, every open set in  $(\mathbb D^d[0,T],\rho^{(1)})$ is also open in $(\mathbb D^d[0,T],\rho_c^{(1)})$. And if $\mathcal K$ is compact in $(\mathbb C^d[0,T],\rho_c^{(1)})$, it is also compact in $(\mathbb D^d[0,T],\rho_c^{(1)})$ and in $(\mathbb D^d[0,T],\rho^{(1)})$.

To construct the approximation of $\zb_n$, we subdivide the time interval $[0,T]$ into $n$ pieces with nodes $t_j^n = Tj/n$, $j=0,1,\cdots,n$. Define the piecewise linear interpolation $\tilde{\zb}_n(t)$ of $\zb_n(t)$ as
\begin{equation}
\tilde{\zb}_n(t)= (1-\gamma _{j}(t))\zb_n(t_{j}^n)+\gamma _{j}(t)\zb_n(t_{j+1}^n),\quad t \in [t_j^n,t_{j+1}^n],
\end{equation}
where $\gamma _{j}(t)=(t-t_{j})n/T\in [0,1]$.

We have the important characterization that $\tilde{\zb}_n$ is exponentially equivalent to $\zb_n$.
\begin{lem}\label{distance between z and y}
For each $\delta >0$,
\begin{equation}
\mathop {\lim \sup }\limits_{n \to \infty } \frac{1}{n}\log\mathbb{P}(\rho^{(1)}(\zb_n,\tilde{\zb}_n)>\delta)=- \infty.
\end{equation}
\end{lem}
The proof of Lemma \ref{distance between z and y} is left in the Appendix.

For given compact sets in $\mathbb C^d[0,T]$, the following quasi-LDP upper bound for $(\tilde{\zb}_n,\nu_n)$ holds.
\begin{lem}\label{Main lem}
Fix step functions $\thetab(t)\in \mathbb{R}^{d}$ and $\alphab(t)\in \mathbb{R}^{D}$. For any $\delta >0$ and compact sets $\mathcal{K} \in \mathbb{C}^d[0,T]$ and $\mathcal{S}\in \mathbb{M}_L[0,T]$, we have
\begin{equation}\label{is+if}
\mathop {\lim \sup }\limits_{n \to \infty } \frac{1}{n}\log\mathbb{P}((\tilde{\zb}_n,\nu_n)\in \mathcal{K} \times \mathcal{S}) \le
-\mathop {\inf}\limits_{(\rb,\nu) \in \mathcal{K} \times \mathcal{S}} \left(I_s^\delta(\rb,\nu,\thetab)+I_f^\delta(\rb, \nu,\alphab)\right),
\end{equation}
where
\begin{equation}
I_{s}^\delta(\rb,\nu,\thetab)= \left\{
\begin{array}{cl}
\int_0^T L^{\delta}_{s}(\rb(t),\dot\rb,n_{\nu}(t,\cdot),\thetab(t))dt,  & d\rb(t)=\dot{\rb}(t)dt,\\
 \infty,  & \text{otherwise},  \\
 \end{array}
 \right.
 \end{equation}
\begin{equation}\label{eq:LsDeltaDef}
 L^{\delta}_{s}(\zb,\betab,\wb,\pb)=\left\langle \betab,\pb \right\rangle - H_{s}^\delta(\zb,\pb,\wb),
 \end{equation}
 \begin{equation}
 H_{s}^\delta(\zb,\pb,\wb)=\mathop {\sup}_{|\xb-\zb|<\delta}H_{s}\left( {\xb,\pb ,\wb} \right),
 \end{equation}
and
\begin{equation}
I_f^\delta(\rb,\nu,\alphab)= \int_0^TS^\delta(\rb(t),n_\nu(t,\cdot),\alphab(t))dt,
\end{equation}
\begin{equation}\label{eq:SDeltaDef}
S^\delta(\zb,\wb,\sigmab)= -\mathop {\sup}_{|\xb-\zb|<\delta}\sum_{i,j=1}^{D}w_iq_{ij}(\zb)\left(e^{\left\langle\sigmab,\eb_{ij}\right\rangle}
-1\right).
\end{equation}
\end{lem}
 Before  the proof we remark that $H_{s}^\delta(\zb,\pb,\wb)$ and $L^{\delta}_{s}(\zb,\betab,\wb,\pb)$ are monotonically increasing and decreasing functions of $\delta$, respectively.
$S^\delta(\zb,\wb,\sigmab)$ is a monotonically decreasing function of $\delta$.
Correspondingly, $I_{s}^\delta(\rb,\nu,\thetab)$ and $I_f^\delta(\rb,\nu,\alphab)$ are  decreasing functionals of $\delta$.
\begin{proof}
We only need to consider absolutely continuous functions $\rb$ on the right hand side of Eq. \eqref{is+if} since
 $I_s^\delta(\rb,\nu,\thetab)+I_f^\delta(\rb, \nu,\alphab)=\infty$ otherwise.  For any $\rb$ and $\nu$, define the sum
\begin{align}
&J_n(\rb,\thetab ,\nu ,\alphab )= \sum\limits_{j = 0}^{n - 1} \Bigg( \left\langle \rb(t_{j + 1}^n) - \rb(t_j^n),\thetab (t_j^n) \right\rangle  \nonumber\\
 &-\int_{t_j^n}^{t_{j+1}^{n}} H_{s}^\delta  \left( \rb(t_j^n),\thetab (t_j^n),n_\nu(t,\cdot) \right)dt
  + \int_{t_j^n}^{t_{j+1}^{n}} S^\delta(\rb(t_j^n),n_\nu(t,\cdot),\alphab(t_j^n))dt \Bigg). \label{eq:Jn}
\end{align}
By Corollary \ref{cor of local lemma} in Appendix, we have
\begin{equation}\label{eq1 of main lemma}
 \mathop{\lim \sup}\limits_{n\to \infty}\frac{1}{n}\log\mathbb{E} \exp\left(nJ_n(\tilde{\zb}_n,\thetab,\nu_n,\alphab)\right)\le 0.
\end{equation}
For $(\tilde{\zb}_n,\nu_n)\in \mathcal{K}\times \mathcal{S}$, it is obvious that
\begin{equation}
{J_n}(\tilde{\zb}_n,\thetab ,\nu_n ,\alphab )-\mathop {\inf }\limits_{(\rb,v) \in {\mathcal{K}}\times \mathcal{S}}{J_n}(\rb,\thetab ,\nu ,\alphab)\ge 0.
\end{equation}
So we have
 $$\exp\left({J_n}(\tilde{\zb}_n,\thetab ,\nu_n ,\alphab )-\mathop {\inf }\limits_{(\rb,\nu) \in {\mathcal{K}}\times \mathcal{S}}{J_n}(\rb,\thetab ,\nu ,\alphab )\right)
\ge1$$
and
\begin{eqnarray*}
&&\mathbb{P}((\tilde{\zb}_n,\nu_n)\in \mathcal{K}\times \mathcal{S})\\
&\le& \mathbb{E} \exp \left\{n\Big[{J_n}(\tilde{\zb}_n,\thetab ,\nu_n ,\alphab )-\mathop {\inf }\limits_{(\rb,\nu) \in {\mathcal{K} \times \mathcal{S}}}{J_n}(\rb,\thetab ,\nu ,\alphab )\Big]\right\}.
\end{eqnarray*}
Combining this with (\ref{eq1 of main lemma}), we get
\begin{equation}\label{eq2 of main lemma}
\mathop {\lim \sup }\limits_{n \to \infty } \frac{1}{n}\log\mathbb{P}((\tilde{\zb}_n,\nu_n)\in \mathcal{K} \times \mathcal{S}) \le -\mathop {\lim \inf }\limits_{n \to \infty }\left(\mathop {\inf }\limits_{(\rb,\nu) \in \mathcal{K} \times \mathcal{S}}{J_n}(\rb,\thetab ,\nu ,\alphab )\right).
\end{equation}

We now represent the sum on the right hand side of (\ref{eq2 of main lemma}) as an integral. Since $\mathcal{K}$ is compact, the absolutely continuous functions $\rb\in \mathcal{K}$ are thus uniformly bounded. Let $V$ be a compact set in $\mathbb{R}^d$ such that
$$ \big\{\zb: \zb=\rb(t)\text{ for some } \rb\in \mathcal{K}\text{ and }t\in [0,T]\big\}\subset V.$$
For step function $\thetab$, let us investigate an interval in which $\thetab$ takes constant value $\thetab_{0}$, say,  the interval $[0,\tau]$ without loss of generality. Then
$$\sum\limits_{j = 0}^{n - 1} \chi_{\{t_{j + 1}^n \leqslant \tau\}} \left\langle \rb(t_{j+1}^n)-\rb(t_j^n),\thetab_{0} \right\rangle  = \int_0^\tau  {\left\langle {\dot\rb,\thetab_{0}} \right\rangle } dt+\epsilon_n,$$
where the error $\epsilon_n$ takes into account the fact that $\tau$ may not match any of $t_j^n$. It goes to zero uniformly for $\rb\in \mathcal{K}$ when $n$ goes to infinity from the bound
$$|\epsilon_n|\le \frac{2T}{n}|\thetab_0|\mathop{\sup}\limits_{\zb\in V}|\zb|.$$

Now $H_{s}^\delta(\zb,\pb,\wb)$ is continuous in $\zb$, $\pb$ and $\wb$ from the continuity of $H_{s}$ on $\xb,\pb$ and $\wb$ and the boundedness of $\lambda_i$, $\thetab$ and $n_\nu(t,\cdot)$ in the current setting.  So we have
$$\left| {H_{s}^\delta  \left( {\rb(t_j^n),\thetab (t_j^n),n_\nu(t,\cdot)} \right) - H_{s}^\delta  \left( {\rb(t),\thetab (t_j^n),n_\nu(t,\cdot)} \right)} \right|,~~t_j^n\le t\le t_{j+1}^n$$
goes to zero uniformly in $j$ for $\rb\in \mathcal{K}$ and $\nu \in \mathcal{S}$ by equi-continuity. Therefore,
\begin{align*}
&\sum_{j=0}^{n-1}\chi_{\{t_{j + 1}^n \leqslant \tau\}}\int_{t_j^n}^{t_{j+1}^{n}}H_{s}^\delta  \left( {\rb(t_j^n),\thetab (t_j^n),n_\nu(t,\cdot)} \right)dt\\
&=\int_0^{\tau}H_{s}^\delta  \left( {\rb(t),\thetab (t_j^n),n_\nu(t,\cdot)} \right) dt+\epsilon_n,
\end{align*}
with $\epsilon_n$ converging to zero uniformly in $(\rb,\nu)\in \mathcal{K} \times \mathcal{S}$.

Similarly we can estimate for the part $S^\delta(\rb(t_j^n),n_\nu(t,\cdot),\alphab(t_j^n))$  and repeat the argument on the finite number of intervals on which $\thetab$ and $\alphab$ are constants. Thanks to the uniformity in $(\rb, \nu) \in \mathcal{K} \times \mathcal{S} $, we obtain$$
\mathop {\lim \inf }\limits_{n \to \infty }\left(\mathop {\inf }\limits_{(\rb,\nu) \in \mathcal{K} \times \mathcal{S}}{J_n}(\rb,\thetab ,\nu ,\alphab )\right) = \mathop {\inf }\limits_{(\rb,\nu) \in \mathcal{K} \times \mathcal{S}}( I_{s}^\delta(\rb,\nu,\thetab)+I_{f}^\delta(\rb,\nu,\alphab)).$$
Together with (\ref{eq2 of main lemma}), the proof is completed.
\end{proof}

Next we show the exponential tightness of the sequence $(\tilde\zb_{n},\nu_n)$. Define the modulus of continuity of a continuous function $\zb$ as
\begin{equation}
V_{\delta}(\zb)= \sup\big\{\|\zb(t)-\zb(s)\|:0 \le s\le t \le T,|t-s|<\delta \big\}
\end{equation}
and the set
\begin{equation}\label{eq:KM}
\mathcal{K}(M)= \bigcap\limits_{m = M}^\infty \left\{\zb \in \mathbb C^d[0,T]: \zb(0)=\zb^0, V_{2^{-m}}(\zb)\le \frac{1}{\log m}\right\}
\end{equation}
for any fixed $M\in \mathbb N$.

\begin{lem}[Exponential tightness for $\tilde \zb_{n}$]\label{exponential tightness}
For each $B>0$, there is a compact set $\mathcal{K} \subset \mathbb{C}^d[0,T]$ such that
$$\mathop{\lim \sup}\limits_{n\to \infty}\frac{1}{n}\log\mathbb{P}(\tilde{\zb}_n \notin\mathcal{K}) \le -B.$$
\end{lem}
\begin{proof}
For any fixed $M\in \mathbb N$,  it is not difficult to see that the set $\mathcal{K}(M)$ is closed and the functions in $\mathcal{K}(M)$ are equicontinuous. Thus $\mathcal{K}(M)$ is compact by the Arzela-Ascoli theorem. If $2^{-m}<T/n$, we have
$$V_{2^{-m-1}}(\tilde{\zb}_n)=\frac{1}{2}V_{2^{-m}}(\tilde{\zb}_n)$$
since $\tilde{\zb}_n$ is piecewise linear. Therefore, to check whether $\tilde{\zb}_n$ is in $\mathcal{K}(M)$, we only need to consider a finite intersection, for values of $m$ up to
$$M(n)= \max\left\{M,\left\lceil {\frac{{\log (n/T)}}{{\log 2}}} \right\rceil \right\}.$$
Using Corollary \ref{cor of lemma1} in Appendix, we have for any $n$ with $M(n)>M$,
\begin{eqnarray*}
\mathbb{P}(\tilde{\zb}_n\notin\mathcal{K}(M))&\le& \sum_{m=M}^{M(n)}\mathbb{P}\left(V_{2^{-m}}(\tilde{\zb}_n)>\frac{1}{\log m}\right)\\
&\le& \sum_{m=M}^{M(n)}\sum_{j=0}^{n-1}\mathbb{P}\left(\mathop{\sup}\limits_{0\le t\le 2^{-m}}|\zb_n(t_j^n+t)-\zb_n(t_j^n)|>\frac{1}{\log M}\right)\\
&\le& n M(n) \cdot  2d \exp\left(-n\frac{c_1}{\log M}\log\left(\frac{2^Mc_2}{\log M}\right)\right)
\end{eqnarray*}
for positive constants $c_{1}$ and $c_{2}$. Thus
$$\mathop{\lim \sup}\limits_{n\to \infty}\frac{1}{n}\log\mathbb{P}(\tilde{\zb}_n \notin \mathcal{K}(M)) \le -c\frac{M}{\log M}$$
for some positive constant $c$ when $M\gg 1$.
\end{proof}

\begin{lem}\label{e tight of v}
The measure space $\mathbb{M}_L[0,T]$  is compact.
\end{lem}
\begin{proof} Since $[0,T] \times \{1,2,\ldots ,D\}$ is compact, $\mathbb{M}_L[0,T]$ is tight. By Prohorov's theorem, $\mathbb{M}_L[0,T]$  is relatively compact.
Let $\nu$ be the limit of any converging sequence $\{\nu_m\}$ in $\mathbb{M}_L[0,T]$. Since $\sum_{i=1}^D \nu_m(dt,i)=dt$ for all $m$, we have $\sum_{i=1}^D \nu(dt,i)=dt$ and thus $\nu(dt,i)\ll dt$. So $\nu$ also belongs to $\mathbb{M}_L[0,T]$. This proves that $\mathbb{M}_L[0,T]$ is compact.
\end{proof}

The straightforward consequence of Lemma \ref{e tight of v} is that  $\nu_n$ is also exponentially tight.

Define the quasi-rate functionals for slow and fast variables corresponding to $I_s$ and $I_f$ in Theorem \ref{second result}
\begin{equation}
I_{s}^\delta(\rb,\nu)= \left\{
\begin{array}{cl}
   \int_0^T \mathop L^{\delta}_{s}(\rb(t),\dot\rb(t),n_{\nu}(t,\cdot))dt, & d\rb(t)=\dot{\rb}(t)dt,  \\
   \infty,  & \text{otherwise},  \\
 \end{array}
 \right.
\end{equation}
\begin{equation}
L^{\delta}_{s}(\zb,\betab,\wb) = \sup\limits_{\pb \in \mathbb{R}^d}L^{\delta}_{s}(\zb,\betab,\wb,\pb),
\end{equation}
and
\begin{equation}
I_{f}^\delta(\rb,\nu)=\int_0^T S^\delta(\rb(t),n_\nu(t,\cdot))dt,
\end{equation}
\begin{equation}\label{Sdelta}
 S^\delta(\zb,\wb)=\mathop{\sup}\limits_{\sigmab \in \mathbb{R}^D}S^\delta(\zb,\wb,\sigmab).
\end{equation}
The definitions of $L^{\delta}_{s}(\zb,\betab,\wb,\pb)$ and $S^\delta(\zb,\wb,\sigmab)$ are referred to \eqref{eq:LsDeltaDef} and \eqref{eq:SDeltaDef}.  We have the following approximation lemmas.
\begin{lem}\label{last lem of upp}
 For any $\epsilon >0$,  the absolutely continuous function $\rb\in \mathbb{C}^d[0,T]$ and $\nu \in \mathbb{M}_L [0,T]$, there exists neighborhood $N_{\rb,\nu} \in \mathbb{C}^d[0,T]\times\mathbb{M}_L[0,T]$ of $(\rb,\nu)$, step functions $\thetab_{\rb\nu} \subset \mathbb{R}^d$ and $\alphab_{\rb\nu} \subset \mathbb{R}^D$, such that for any $\qb \in N_{\rb}$ and $\mu \in N_\nu$ which are both absolutely continuous, we have
 $$I_{s}^\delta(\qb,\mu,\thetab_{\rb\nu})+I_{f}^\delta(\qb,\mu,\alphab_{\rb\nu})\ge I_{s}^\delta(\rb,\nu)+I_{f}^\delta(\rb,\nu)-\epsilon.$$
\end{lem}
\begin{lem}\label{last lem of upp 2}
For any pair $(\rb,\nu)\in \mathbb{C}^d[0,T]\times \mathbb{M}_L[0,T]$ and $M_0 >0$, if $\rb$  is not absolutely continuous, there exists neighborhood $N_{\rb,\nu} \in \mathbb{C}^d[0,T]\times\mathbb{M}_L[0,T]$ of $(\rb,\nu)$ and step functions $\thetab_{\rb\nu} \in \mathbb{R}^d$ and $\alphab_{\rb\nu} \in \mathbb{R}^D$,  such that for any $(\qb,\mu) \in N_{\rb,\nu}$, we have
 $$I_{s}^\delta(\qb,\mu,\thetab_{\rb\nu})+I_{f}^\delta(\qb,\mu,\alphab_{\rb\nu})\ge M_0.$$
\end{lem}
The Lemmas \ref{last lem of upp} and \ref{last lem of upp 2} are direct consequences of  Lemmas \ref{lower semicontinuity of I+F} and \ref{step function} in the appendix.

Simply denote the product metric $\rho^{(1)}\times \rho^{(2)}$ on $\mathbb D^d[0,T]\times\mathbb{M}_L[0,T]$ as $d(\cdot,\cdot)$  and define the sets
\begin{equation}\label{eq:PhiKDef}
\Phi(K)=\left\{(\rb,\nu)\in \mathbb{D}^d[0,T]\times \mathbb{M}_L[0,T]:I_{s}(\rb,\nu)+I_{f}(\rb,\nu)\le K\right\}
\end{equation}
and
\begin{equation}\label{eq:PhiKDeltaDef}
\Phi^\delta(K)=\left\{(\rb,\nu)\in \mathbb{D}^d[0,T]\times \mathbb{M}_L[0,T]:I_{s}^\delta(\rb,\nu)+I_{f}^\delta(\rb,\nu)\le K\right\}.
\end{equation}
We have the following characterization for $\Phi(K)$ and $\Phi^\delta(K)$.
\begin{lem}\label{compactness of I+F}
For any $K>0$, the level sets $\Phi(K)$ and $\Phi^\delta(K)$ defined in \eqref{eq:PhiKDef} and \eqref{eq:PhiKDeltaDef} are compact sets.
\end{lem}
\begin{proof}By Lemma \ref{e tight of v}, $\mathbb{M}_L[0,T]$ is a compact set. By Lemma \ref{uniformly absolute continuity}, the functions $\rb\in \Phi(K)$ are equicontinuous. Combining with the fact that $\rb(0)=\zb^{0}$, we have that $\Phi(K)$ is pre-compact. By Lemma \ref{lower semicontinuity of I+F}, $I_s(\rb,\nu)+I_{f}(\rb,\nu)$ is lower semicontinuous. Consequently, $\Phi(K)$ is closed and thus compact. The proof for $\Phi^\delta(K)$ is similar.
\end{proof}

\begin{pro}\label{main Proposition}
For each $K>0$, $\delta >0$ and $\epsilon >0$, $$
\mathop {\lim \sup }\limits_{n \to \infty } \frac{1}{n}\log\mathbb{P}\left(d\left((\tilde{\zb}_n,\nu_n),\Phi^{\delta}(K)\right)>\epsilon\right)\le -(K-\epsilon).$$
\end{pro}

\begin{proof} From the exponential tightness, we can find a compact set $\mathcal{K}^N \in \mathbb{C}^d[0,T]$ for each $N>0$ such that
 $$\mathop {\lim \sup }\limits_{n \to \infty } \frac{1}{n}\log\mathbb{P}(\tilde{\zb}_n \notin \mathcal{K}^N)\le -N.$$
Define the set
 $$\mathcal{K}^{N,\epsilon} = \overline{\left\{(\rb,\nu)\in \mathbb C^d[0,T]\times\mathbb{M}_L[0,T]: d\left((\rb,\nu),\Phi^{\delta}(K)\right)>\epsilon\right\}} \cap (\mathcal{K}^N\times \mathbb{M}_L[0,T]).$$
For any $(\rb,\nu) \in \mathcal{K}^{N,\epsilon}$, we can find the neighborhood  $N_{\rb,\nu}$ either satisfying  Lemma \ref{last lem of upp} if $\rb$ is absolutely continuous, or satisfying Lemma \ref{last lem of upp 2} if $\rb$ is not absolutely continuous. This forms a covering of $\mathcal{K}^{N,\epsilon}$. By compactness, we can choose a finite subcover $\{N_{\rb_i,\nu_j}\}_{i,j}$ for $\mathcal{K}^{N,\epsilon}$. Define
 $$\mathcal K_{ij}= \overline {N_{\rb_i,\nu_j} \cap \mathcal{K}^{N,\epsilon}}.$$
Applying Lemma \ref{Main lem}, Lemma \ref{last lem of upp}, Lemma \ref{last lem of upp 2} and letting $M_0$ in Lemma \ref{last lem of upp 2} larger than $K$, we have for any $i,j$,
 $$\mathop {\lim \sup }\limits_{n \to \infty } \frac{1}{n}\log\mathbb{P}((\tilde{\zb}_n,\nu_n)\in \mathcal {K}_{ij})\le
-\left(K-\epsilon\right). $$

 Then we have
\begin{align*}
&\mathop {\lim \sup }\limits_{n \to \infty } \frac{1}{n}\log \mathbb{P}\left(d\left((\tilde{\zb}_n,\nu_n),{\Phi^\delta}(K)\right)> \epsilon \right)\\
\le&\mathop {\lim \sup }\limits_{n\to \infty } \frac{1}{n}\log\left[\mathbb{P}(\tilde{\zb}_n \notin \mathcal{K}^N)+\sum\limits_{i,j }\mathbb{P}((\tilde{\zb}_n,\nu_n)\in \mathcal{K}_{ij})\right]\\
\le&-\min\left\{N,K-\epsilon \right\}.
\end{align*}
Choosing $N$ large enough, we complete the proof.
\end{proof}

We are now ready to establish the upper bound.
\begin{lem}\label{phi and  phi^delta}
Given $K>0$ and $\epsilon>0$, there exist $\delta>0$ such that
$$ \Phi^\delta(K-\epsilon) \subset \{(\rb,\nu):d((\rb,\nu),\Phi(K))\le \epsilon\}.$$
\end{lem}

\begin{proof}
Prove  by contradiction. If the claim is false, we can choose
$$\delta_i  \downarrow 0, \quad \quad (\rb_i,\nu_i) \in \Phi^{\delta_i}(K-\epsilon), \quad i=1,2,\cdots$$
such that
\begin{equation}\label{eq:ContraEq}
d((\rb_i,\nu_i),\Phi(K))\ge \epsilon, \quad \forall i.
\end{equation}

By definition of $I_s^\delta(\rb,\nu)$ and $I_f^\delta(\rb,\nu)$, we have the monotonicity $I_{s}^{\delta}(\rb,\nu)\le I_{s}^{\delta'}(\rb,\nu)$ and $I_{f}^{\delta}(\rb,\nu)\le I_{f}^{\delta'}(\rb,\nu)$ when $\delta \ge \delta'\ge 0$. Thus the sets $\Phi^{\delta_i}(K-\epsilon)$ are monotonically decreasing as $\delta_i \downarrow 0$, and $(\rb_i,\nu_i)$ are contained in the set $\Phi^{\delta_1}(K-\epsilon)$ which is compact by Lemma \ref{compactness of I+F}.
So there exists a subsequence converging to  $(\rb_0,\nu_0)$.  With Lemma \ref{lower semicontinuity of I+F} in the Appendix we have for
each $j$
\begin{eqnarray*}
I_{s}^{\delta_j}(\rb_0,\nu_0)+I_{f}^{\delta_j}(\rb_0,\nu_0) &\le& \mathop {\lim\inf}_{i\to \infty}\left(I_{s}^{\delta_j}(\rb_i,\nu_i)+I_{f}^{\delta_j}(\rb_i,\nu_i)\right)\\
&\le& \mathop {\lim\inf}_{i\to \infty}\left(I_{s}^{\delta_i}(\rb_i,\nu_i)+I_{f}^{\delta_i}(\rb_i,\nu_i)\right)\\
&\le& K-\epsilon.
\end{eqnarray*}
The monotone convergence theorem gives
\begin{eqnarray*}
I_{s}(\rb_0,\nu_0)+I_{f}(\rb_0,\nu_0) &=& \mathop {\lim}_{j\to \infty}I_{s}^{\delta_j}(\rb_0,\nu_0)+I_{f}^{\delta_j}(\rb_0,\nu_0)\\
&\le& K-\epsilon.
\end{eqnarray*}
So $(\rb_0,\nu_0) \in \Phi(K)$. For sufficiently large $i$, $d((\rb_0,\nu_0),(\rb_i,\nu_i))\le \epsilon$. This contradicts with \eqref{eq:ContraEq}.
\end{proof}

\begin{thm}
For each closed set $F \subset \mathbb{D}^d[0,T] \times \mathbb{M}_L[0,T]$,
$$\mathop{\lim \sup}\limits_{n \to \infty} \frac{1}{n} \log\mathbb{P}((\zb_n,\nu_n) \in F)\le  -\mathop {\inf}\limits_{(\rb,\nu) \in F} \left(I_s
(\rb,\nu)
+I_{f}(\rb,\nu)\right).$$
\end{thm}
\begin{proof} Suppose $\mathop {\inf}_{(\rb,\nu) \in F} \left(I_s(\rb,\nu)
+I_{f}(\rb,\nu)\right)=K< \infty $.  Since $F$ and $\Phi(K-\epsilon)$ are both closed sets, we assume the distance between them is $\eta_0>0$. For any $\eta \le \eta_0$,
\begin{eqnarray}\notag
&&\mathbb{P}((\zb_n,\nu_n) \in F)\\\notag
&\le& \mathbb{P}\left(d\left((\tilde \zb_n, \nu_n),F\right)\le \frac{\eta}{2}\right)+\mathbb{P} \left(d(\tilde{\zb}_n,\nu_n),(\zb_n,\nu_n)) \ge \frac {\eta}{2}\right) \\
\label{equation up 0}
&\le&  \mathbb{P}\left(d\left((\tilde{\zb}_n,\nu_n),\Phi(K-\epsilon)\right)\ge \frac{\eta}{2}\right)+\mathbb{P} \left(\rho^{(1)}((\tilde{\zb}_n,\zb_n) \ge \frac {\eta}{2
}\right)
\end{eqnarray}
By Lemma \ref{phi and  phi^delta}, we can choose $\delta$ and $\eta$ small enough such that
$$d\big((\tilde{\zb}_n,\nu_n),\Phi(K-\epsilon)\big)\ge \frac{\eta}{2} \quad\text{implies}\quad d\left((\tilde{\zb}_n,\nu_n),{\Phi^\delta}(K-\epsilon-\eta/4)\right)\ge \frac{\eta}{4}.$$
From Proposition \ref{main Proposition} we have
\begin{eqnarray}\notag
&&\mathop{\lim \sup}\limits_{n \to \infty} \frac{1}{n} \log \mathbb{P}\left(d\left((\tilde{\zb}_n,\nu_n),\Phi(K-\epsilon)\right)\ge \frac{\eta}{2}\right)\\ \notag
&\le& \mathop{\lim \sup}\limits_{n \to \infty} \frac{1}{n} \log \mathbb{P}
\left(d\left((\tilde{\zb}_n,\nu_n),{\Phi^\delta}(K-\epsilon-\eta/4)\right)\ge \frac{\eta}{4}\right)\\
&\le&-(K-\epsilon-\eta/2).\label{equation up 1}
\end{eqnarray}
Combining \eqref{equation up 0}, \eqref{equation up 1} and Lemma \ref{distance between z and y} for $\delta = \eta/4$,  we obtain
$$\mathop{\lim \sup}\limits_{n \to \infty} \frac{1}{n} \log\mathbb{P}((\zb_n,\nu_n) \in F)\le -(K-\epsilon-\eta/2).$$
The case for
$$\mathop {\inf}\limits_{(\rb,\nu) \in F} \left(I_s(\rb,\nu)
+I_{f}(\rb,\nu)\right)=\infty$$
can be established similarly by choosing $K$ arbitrarily large.
\end{proof}

\subsection{Lower bound}
The proof of the lower bound is based on the change of measure formula.
From \cite{Dembo1993}, it suffices to prove that for any $(\rb,\nu) \in \mathbb{D}^d[0,T]\times \mathbb{M}_L[0,T]$ and arbitrarily small $\epsilon>0$ we have
\begin{equation}
\mathop{\lim \inf}\limits_{n\to \infty}\frac{1}{n}\log\mathbb{P}\left(\zb_n\in N_\epsilon(\rb),\nu_n\in N_\epsilon(\nu)\right) \ge - \left(I_s(\rb,\nu)+I_{f}(\rb,\nu)\right),
\end{equation}
where $N_\epsilon(\rb)$ is the $\epsilon$-neighborhood of $\rb$ in $\mathbb{D}^d[0,T]$ with metric $\rho^{(1)}$, and
$N_\epsilon(\nu)$ is the $\epsilon$-neighborhood of $\nu$ in $\mathbb{M}_L[0,T]$ with metric $\rho^{(2)}$. For given $\rb \in \mathbb{D}^d[0,T]$ and $\nu \in \mathbb{M}_L[0,T]$, if $\rb$ is not absolutely continuous, $I_s(\rb,\nu)+I_{f}(\rb,\nu)=\infty$, thus nothing needs to be proved. Below we will exclude this case. For convenience, we further assume that $n_\nu(t,i)$ is continuous in $t$, and the case that $n_\nu(t,i)$ is not continuous will be discussed in Theorem \ref{final theorem} in this section. To prove the lower bound, we perform the following steps. Firstly, we approximate $\rb$ by a piecewise linear path $\yb$, and the occupation measure $\nu$ by $\pi\in \mathbb M_L[0,T]$ with $n_{\pi}(t,\cdot)$ piecewise constant in $t$. Secondly, we construct new processes $\bar \zb_n$ and $\bar \xi_n$ with occupation measure $\bar\nu_{n}$  such that
\begin{equation}\label{eq:Plimz}
\mathbb{P}-\lim_{n \to \infty}\rho_c^{(1)}(\bar \zb_n,\yb)=0,\quad  \mathbb{P}-\lim_{n \to \infty}\rho^{(2)}(\bar \nu_n,\pi)=0,
 \end{equation}
 where the notation $\mathbb P-\lim$ means the convergence in probability.  Moreover, we ask  $\bar \zb_n$ and the jump rates of $\bar \xi_n$ satisfy the conditions required by Lemmas \ref{up 1} and \ref{up 2}. Finally, based on the change of measure formula related to $(\zb_n,\xi_n)$ and $(\bar \zb_n,\bar \xi_n)$, we get the limit and the proof is then finished.

As promised in the above procedure,  we approximate $\rb$ by a  path $\yb$ first. For a given $J$, define $\Delta = T/J$ and let $t_m = m\Delta$. On each interval $[t_m,t_{m+1}]$, define $\Delta \rb_m=\rb(t_{m+1})-\rb(t_m)$. Take $\mub^m = \{\mu_i^m,i=1,\ldots, S\}$ so as to satisfy
\begin{equation}
\sum_{i=1}^S \mu_i^m\ub_i=\frac{\Delta \rb_m}{\Delta} \quad \mbox{and}\quad \mu^m_i\ge 0.
\end{equation}
If $\Delta \rb_m$ are in the positive cone generated by the $\{\ub_i\}$ for all $m$, such a choice of $\mub^m$ is possible. If at least one of $\Delta \rb_m$ is not in the positive cone generated by the $\{\ub_i\}$ , it is easy to check that  for all $\nu \in \mathbb {M}[0,T]$, $I_s(\rb,\nu)=+\infty$ (see the Remark of Lemma 5.21 in \cite{Shwartz1995}) and nothing needs to be proved.

Now we construct the piecewise linear interpolation $\yb$ of $\rb$ such that $\yb(t_{0})=\rb(t_{0})$ and in each time interval $[t_m,t_{m+1}]$
\begin{equation}
\frac{d}{dt}\yb(t)=\sum_{i=1}^S \mu_i^m\ub_i.
\end{equation}
Thus $\yb(t_{m})=\rb(t_{m})$ for each $m$. For any $\epsilon>0$, we can choose $J$ large enough such that
 $$\rho_c^{(1)}(\yb,\rb)<\epsilon/4.$$

Define the sets
\begin{eqnarray}\label{Sm}
\mathcal S=\Big\{(\eta,\psi)\big| {\eta}=(\eta_{ij})_{D\times D}, ~\eta_{ij}> 0;~ \psi\in \Delta_D;
\sum_{i=1}^D\psi_i\sum_{j=1}^D\eta_{ij}{\eb}_{ij}=\boldsymbol 0
\Big\}
\end{eqnarray}
and
\begin{eqnarray}\label{Ky}
\mathcal K_{\betab}=\Big\{\mub\in \mathbb R^S:\mu_{i}\ge 0,\sum_{i=1}^{S} \mu_{i}\ub_i=\betab\Big\}.
\end{eqnarray}
We remark that the sets $\mathcal{S}$ and $\mathcal{K}_{\betab}$ here have nothing to do with the definitions in the proof of upper bound.

\begin{lem}\label{up 1}
For any $\epsilon>0$ and  large enough $J$, there exists a further subdivision of time interval $[t_m,t_{m+1}]$ for each $m\in\{0,1,\cdots,J-1\}$ (i.e., $t_m=t_{m0}<t_{m1}<\cdots<t_{mK_m}=t_{m+1}$) and related  $(\eta^{mk},\psi^{mk})\in \mathcal S$ $(m=0,1,\cdots,J-1; k=0,1,\cdots,K_m-1)$, such that
\begin{eqnarray*}
&&\sum_{m=0}^{J-1}\sum_{k=0}^{K_m-1}\int_{t_{mk}}^{t_{m,k+1}}\sum_{i=1}^D \psi^{mk}_i\sum_{j=1}^D\Big(\eta_{ij}^{mk}\log \frac{\eta_{ij}^{mk}}{q_{ij}(\yb(t))}+q_{ij}(\yb(t))-\eta_{ij}^{mk}\Big)dt\\
&\le& I_f(\rb,\nu)+\epsilon.
\end{eqnarray*}
and
\begin{equation}\label{eq:PsiNnuDiff}
\|\psi^{mk}-n_\nu(t,\cdot)\|<\epsilon/(4DT)
\end{equation}
for all $t \in [t_{mk},t_{m,k+1})$,  $k=0,1,\cdots,K_m-1$ and $m=0,1,\cdots,J-1$.
\end{lem}
The proof of  Lemma \ref{up 1} can be found in the Appendix.

We then define the measure $\pi\in \mathbb M_L[0,T]$ such that $\pi(dt,i)=n_{\pi}(t,i)dt$ and
$$n_\pi(t,i):= \psi^{mk}_i,\quad t \in [t_{mk},t_{m,k+1})$$
for $m=0,1,\cdots,J-1$ and $k=0,1,\cdots,K_m-1$. With this choice $n_{\pi}(t,\cdot)$ is piecewise constant and
$$\rho^{(2)}(\pi,\nu)<\epsilon/4.$$

We take the frequently used notation $\lambda^{\pi}_{i}$ in later text as the expectation of $\lambda_{i}$ with respect to the distribution $n_{\pi}$
\begin{equation}
\lambda^{\pi}_{i}(\yb(s)) = \sum_{j=1}^D\lambda_i(\yb(s),j)n_{\pi}(s,j).
\end{equation}

\begin{lem}\label{up 2}
For any $\epsilon>0$ and  large enough $J$, define $\betab_m = \Delta r_m/\Delta$, then there exists  $\mub^m\in K_{\betab_m}$ such that
\begin{align*}
&\sum_{m=0}^{J-1}\int_{t_m}^{t_{m+1}}\sum_{i=1}^S\Big(\lambda^{\pi}_{i}(\yb(t))-\mu_i^m+\mu_i^m\log \frac{\mu_i^m}{\lambda^{\pi}_{i}(\yb(t))}\Big)dt  \\
& \le I_s(\rb,\nu)+\epsilon.
\end{align*}
\end{lem}

The proof of Lemma \ref{up 2} can be found in the Appendix.

With the constructed matrices $\{\eta^{mk}\}$ in Lemma \ref{up 1}, we define the process $\bar\xi_n$ with jump rate $n\eta_{ij}(t)$ where $\eta_{ij}(t)=\eta_{ij}^{mk},~t\in[t_{mk},t_{m,k+1})$. Similarly, we take $\mub^m$ constructed from Lemma \ref{up 2} and define $\bar \zb_n$
with jump rate
$$n\mu_i(t)\frac{\lambda_{i}(\bar \zb_n(t),\bar\xi_n(t))}{\lambda^{\pi}_{i}(\yb(t))}$$
for its $i$th component, where $\mu_i(t)$ is piecewise constant and  $\mu_i(t)=\mu_i^m$ for $t\in [t_m,t_{m+1})$.

We have the following convergence result for the constructed approximations for $\pi$ and $\yb$.
\begin{lem}\label{covergence of tilde v} Convergence of the approximation $\bar\nu_{n}$
 $$ \mathbb{P}-\lim_{n \to \infty}\rho^{(2)}(\bar \nu^n,\pi)=0.$$
\end{lem}
\begin{lem}\label{covergence of tilde z} Convergence of the approximation $\bar\zb_{n}$
$$\mathbb{P}-\lim_{n \to \infty}\rho_c^{(1)}(\bar \zb^n,\yb)=0.$$
\end{lem}
The proof of Lemmas  \ref{covergence of tilde v}, \ref{covergence of tilde z} will be given in the Appendix.

As we have finished the construction of $\bar \zb_n$ and $\bar \xi_n$, we now perform the change of measure. Denote $\mathbb{Q}_{n}$ and $\bar{\mathbb{Q}}_{n}$ the distributions of $(\zb_n(t),\xi_n(t))_{t\le T}$ and  $(\bar \zb_n(t),\bar \xi_n(t))_{t\le T}$, respectively. We have
\begin{eqnarray}\notag
&&\frac{d\mathbb{Q}_n}{d\bar{\mathbb{Q}}_n}(\bar \zb_n,\bar \xi_n)\\ \notag
&=&\exp\left\{-\int_0^Tn\sum_{i=1}^d\left(\lambda_i(\bar \zb_n(t),\bar \xi_n(t))-{\mu_i(t)\frac{{{\lambda _i}({{\bar z}_n}(t),{\bar \xi_n}(t))}}{\lambda^{{\pi}}_{i}(\yb(t))}}\right)\right.dt\\ \notag
&&-\int_0^T\sum_i \log \frac{{{\mu _i(t^-)}}}{{\lambda^{\pi}_{i}(\yb(t^{-}))}}dY_t^i-\int_0^Tn\sum_{i,j=1}^D\Big(q_{ij}(\bar \zb_n(t))-\eta_{ij}(t)\Big)dt\\\notag
&&-\left.\int_0^T\sum_{i,j} \log \frac{\eta_{ij}(t^-)}{q_{ij}(\bar \zb_n(t^-))}dM_t^{ij}
\right\}\\\label{formula}
& := &e^{B(\bar \zb_n,\bar \xi_n)},
\end{eqnarray}
where $Y_t^i$ is the counting process induced by $\bar \zb_n(t)$ that will increase by one each time when a jump occurs in the $\ub_i$ direction and $M_t^{ij}$ is the counting process induced by $\bar \xi_n(t)$ that will increase by one each time when
a jump occurs from state $i$ to state $j$. The next lemma shows that the expectation of $B(\bar \zb_n,\bar \xi_n)$ in the exponent becomes simple in the limit $n\rightarrow\infty$.


\begin{lem}\label{section 7 of Adam's book}
\begin{align}\notag
&\lim_{n\to \infty}\mathbb{E}_{\bar{\mathbb{Q}}_n}\int_0^T\lambda_i(\bar \zb_n(t),\bar \xi_n(t))-{\mu_i(t)\frac{{{\lambda_i}({{\bar z}_n}(t),{\bar \xi_n}(t))}}{\lambda^{{\pi}}_{i}(\yb(t))}}dt\\\label{convergence1}
=&\sum_{m=0}^{J-1}\int_{t_{m}}^{t_{m+1}}\lambda^{{\pi}}_{i}(\yb(t))-\mu_i^mdt.
\end{align}
\begin{align}\notag
&\lim_{n\to \infty}\frac1n\mathbb{E}_{\bar{\mathbb{Q}}_n}\int_0^T\sum_i \log \frac{{{\mu_i(t^-)}}}{\lambda^{{\pi}}_{i}(\yb(t))}dY_t^i \\\label{convergence2}
=&\sum_{m=0}^{J-1}\int_{t_{m}}^{t_{m+1}}\sum_{i=1}^S\mu_i^m \log \frac{\mu_i^m}{\lambda^{{\pi}}_{i}(\yb(t))}dt.
\end{align}
\begin{align}\notag
&\lim_{n\to \infty}\frac{1}{n}\mathbb{E}_{\bar{\mathbb{Q}}_n}\int_0^T\sum_{i,j} \log \frac{\eta_{ij}(t^-)}{q_{ij}(\bar \zb_n(t^-))}dM_t^{ij} \\\label{convergence3}
=&\sum_{m=0}^{J-1}\sum_{k=0}^{K_m-1}\int_{t_{mk}}^{t_{m,k+1}}\sum_{i=1}^D n_{\pi}(t,i)\sum_{j=1}^D{\eta_{ij}^{mk}}\log \frac{{\eta_{ij}^{mk}}}{q_{ij}(\yb(t))}dt.
\end{align}
\end{lem}

 The proof of Lemma \ref{section 7 of Adam's book} is based on the ideas in proving Lemma 5.52 and Lemma 8.70 in \cite{Shwartz1995}.
\begin{proof}
Since $\mu_i(t)$ is a step function and constant in $[t_{m},t_{m+1})$, to prove \eqref{convergence1}, we just need to prove for each $m$
\begin{align}
&\lim_{n\to \infty}\mathbb{E}_{\bar{\mathbb{Q}}_n}\int_{t_{m}}^{t_{m+1}}\lambda_i(\bar \zb_n(t),\bar \xi_n(t))-\mu_i^m\frac{{{\lambda_i}({{\bar z}_n}(t),{\bar \xi_n}(t))}}{\lambda^{{\pi}}_{i}(\yb(t))}dt\notag\\
=&\int_{t_{m}}^{t_{m+1}}\lambda^{{\pi}}_{i}(\yb(t))-\mu_i^mdt.
\end{align}
Define
$$N_\epsilon(\yb):=\{\zb \in\mathbb D^d[0,T]: \rho_c^{(1)}(\zb,\yb)\le \epsilon \}.$$
We have
\begin{align}\notag
&\mathbb{E}_{\bar{\mathbb{Q}}_n}\int_{t_{m}}^{t_{m+1}}\lambda_i(\bar{\zb}_n(t),\bar \xi_n(t))-\mu_i^m\frac{{{\lambda_i}({{\bar \zb}_n}(t),{\bar \xi_n}(t))}}{\lambda^{{\pi}}_{i}(\yb(t))}dt\\\notag
=&\mathbb{E}_{\bar{\mathbb{Q}}_n} \chi_{\left\{\bar \zb_n\in N_\epsilon(\yb)\right\}} \int_{t_{m}}^{t_{m+1}}\lambda_i(\bar \zb_n(t),\bar \xi_n(t))-\mu_i^m\frac{{{\lambda_i}({{\bar \zb}_n}(t),{\bar \xi_n}(t))}}{\lambda^{{\pi}}_{i}(\yb(t))}dt\\ \label{two terms}
&+\mathbb{E}_{\bar{\mathbb{Q}}_n} \chi_{\left\{\bar \zb_n\notin N_\epsilon(\yb)\right\}} \int_{t_{m}}^{t_{m+1}}\lambda_i(\bar \zb_n(t),\bar \xi_n(t))-\mu_i^m\frac{{{\lambda_i}({{\bar \zb}_n}(t),{\bar \xi_n}(t))}}{\lambda^{{\pi}}_{i}(\yb(t))}dt.
\end{align}
By Lemma \ref{covergence of tilde z}, the second term on the right hand side of \eqref{two terms} tends to zero as $n \to \infty$. Next let us estimate the first term.

By Assumption \ref{ass:main}, we have
\begin{align}
\left|\frac{\lambda_i(\xb',j)}{\lambda^{{\pi}}_{i}(\xb)}-\frac{\lambda_i(\zb,j)}{\lambda^{{\pi}}_{i}(\zb)}\right|&\le \left|\frac{\lambda_i(\xb',j)}{\lambda^{{\pi}}_{i}(\xb)}-\frac{\lambda_i(\xb',j)}{\lambda^{{\pi}}_{i}(\zb)}\right|+\left|\frac{\lambda_i(\xb',j)-\lambda_i(\zb,j)}{\lambda^{{\pi}}_{i}(\zb)}\right|\notag\\
&\le \Lambda^3L \|\xb-\zb\|+\Lambda L\|\xb'-\zb\| \label{zb1zb2zb}
\end{align}
for any $\zb,\xb,\xb'\in \mathbb{R}^d$, $i\in\{1,\ldots,S\}$ and $j \in \{1,2\ldots, D\}.$

Now take an integer $N$ and divide $[t_{m},t_{m+1}]$ into $L$ pieces. Define $\tau_l=t_{m}+l (t_{m+1}-t_{m})/N$ for $l=0,\ldots,N$. Since $\yb$ is continuous in $[0,T]$, we can choose $N$ large enough such that
$$\mathop{\sup}\limits_{t\in [\tau_l,\tau_{l+1}]}\|\yb(t)-\yb(\tau_l)\|\le \epsilon \quad \mbox{for any } l \in \{0,\cdots,N-1\}.$$
By \eqref{eq:LambdaLip} and \eqref{zb1zb2zb}, we have
\begin{align*}
 &\chi_{\left\{\bar \zb_n\in N_\epsilon(\yb)\right\}} \int_{t_{m}}^{t_{m+1}}\lambda_i(\bar \zb_n(t),\bar \xi_n(t))-\mu_i^m\frac{{{\lambda_i}({{\bar z}_n}(t),{\bar \xi_n}(t))}}{\lambda^{{\pi}}_{i}(\yb(t))}dt\\
\le &  \chi_{\left\{\bar \zb_n\in N_\epsilon(\yb)\right\}} \sum_{l=0}^{N-1}\int_{\tau_l}^{\tau_{l+1}}\lambda_i(\yb(\tau_l),\bar \xi_n(t))-\mu_i^m\frac{{{\lambda_i}(\yb(\tau_l),{\bar \xi_n}(t))}}{\lambda^{{\pi}}_{i}(\yb(\tau_l))}+C \epsilon dt,
\end{align*}
where $C=(2+2\Lambda+\Lambda^3)L$. So we have
\begin{align*}
&\lim_{n\to \infty}\mathbb{E}_{\bar{\mathbb{Q}}_n}\int_{t_{m}}^{t_{m+1}}\lambda_i(\bar \zb_n(t),\bar \xi_n(t))-\mu_i^m\frac{{{\lambda_i}({{\bar z}_n}(t),{\bar \xi_n}(t))}}{\lambda^{{\pi}}_{i}(\yb(t))}dt\\
\le& \sum_{l=0}^{N-1}\int_{\tau_l}^{\tau_{l+1}}\sum_{j=1}^D\lambda_i(\yb(\tau_l),j)n_{\pi}(t,j)-\mu_i^m+C \epsilon dt\\
\le&\sum_{l=0}^{N-1}\int_{\tau_l}^{\tau_{l+1}}\sum_{j=1}^D\lambda_i(\yb(t),j)n_{\pi}(t,j)-\mu_i^m+C_1 \epsilon dt\\
=&\int_{t_{m}}^{t_{m+1}}\lambda^{{\pi}}_{i}(\yb(t))-\mu_i^mdt+C_1(t_{m+1}-t_{m})\epsilon
\end{align*}
by ergodicity of the process $\bar\xi_n$, where $C_1=(3+2\Lambda+\Lambda^3)L$. Similarly, we can also obtain
\begin{align*}
&\lim_{n\to \infty}\mathbb{E}_{\bar{\mathbb{Q}}_n}\int_{t_{m}}^{t_{m+1}}\lambda_i(\bar \zb_n(t),\bar \xi_n(t))-\mu_i^m\frac{{{\lambda_i}({{\bar z}_n}(t),{\bar \xi_n}(t))}}{\lambda^{{\pi}}_{i}(\yb(t))}dt\\
\ge &\int_{t_{m}}^{t_{m+1}}\lambda^{{\pi}}_{i}(\yb(t))-\mu_i^mdt-C_1(t_{m+1}-t_{m})\epsilon.
\end{align*}
So we finish the proof for \eqref{convergence1}.

To prove \eqref{convergence2}, we first assume that $\lambda_i(\xb)$ are constant functions.
In $[t_{m},t_{m+1})$, the number of jumps $\bar \zb_n$ makes in each direction $\ub_i/n$ are independent Poisson random variables with mean $n\mu_i^m(t_{m+1}-t_{m})$. So
\begin{align}\notag
&\lim_{n\to \infty}\frac1n\mathbb{E}_{\bar{\mathbb{Q}}_n}\int_{t_{m}}^{t_{m+1}}\sum_i \log \frac{{{\mu_i(t^-)}}}{\lambda^{{\pi}}_{i}(\yb(t))}dY_t^i \\
=&\int_{t_{m}}^{t_{m+1}}\sum_{i=1}^S\mu_i^m \log \frac{\mu_i^m}{\lambda^{{\pi}}_{i}(\yb(t))}dt.
\end{align}
For general $\lambda_i$, we can use the technique for proving \eqref{convergence1} by dividing the interval $[t_{m},t_{m+1}]$ into small pieces and approximating \eqref{convergence2} by Riemann sums.

For \eqref{convergence3}, again we first assume that $q_{ij}(\xb)$ are constant functions. In $[t_{mk},t_{m,k+1}]$, the number of jumps $\bar \xi_n$ makes in each direction $\eb_{ij}$ are independent Poisson random variables with mean $n\cdot n_{\pi}(t_{mk},i)\eta_{ij}^{mk}(t_{m,k+1}-t_{mk})$. So
\begin{align}\notag
&\lim_{n\to \infty}\frac{1}{n}\mathbb{E}_{\bar{\mathbb{Q}}_n}\int_{t_{mk}}^{t_{m,k+1}}\sum_{i,j} \log \frac{\eta_{ij}(t^-)}{q_{ij}(\bar \zb_n(t^-))}dM_t^{ij} \\
=&\int_{t_{mk}}^{t_{m,k+1}}\sum_{i=1}^D n_{\pi}(t,i)\sum_{j=1}^D{\eta_{ij}^{mk}}\log \frac{{\eta_{ij}^{mk}}}{q_{ij}(\yb(t))}dt.
\end{align}
For general $q_{ij}$,  we consider separate cases $\left\{\bar \zb_n\in N_\epsilon(\yb)\right\}$ and $\left\{\bar \zb_n\notin N^c_\epsilon(\yb)\right\}$ as in \eqref{two terms}. Similar as proving \eqref{convergence1}, we can get the limit \eqref{convergence3}.
\end{proof}
\begin{lem}\label{main lemma}
For given $\rb \in \mathbb{D}^d[0,T]$ and $\nu \in \mathbb{M}_L[0,T]$, assume that $\rb$ is absolutely continuous and $n_\nu(t,\cdot)$ is continuous in $t$. Then for arbitrarily small $\epsilon>0$  we have
\begin{eqnarray*}
\mathop{\lim \inf}\limits_{n\to \infty}\frac{1}{n}\log\mathbb{P}\left(\zb_n\in N_\epsilon(\rb),\nu_n\in N_\epsilon(\nu)\right) \ge - \left(I_s(\rb,\nu)+I_{f}(\rb,\nu)\right).
\end{eqnarray*}
\end{lem}
\begin{proof}
By Eq. \eqref{formula} and Jensen's inequality, for any $\epsilon>0$
\begin{align}\notag
&\mathbb{P}\left(\zb_n\in N_\epsilon(\rb),\nu_n\in N_\epsilon(\nu)\right) \\\notag
\ge& \mathbb{P}\left(\zb_n\in N_{\epsilon/2}(\yb),\nu_n\in N_{\epsilon/2}({\pi})\right) \\\notag
=&\mathbb{E}_{\bar{\mathbb{Q}}_{n}}\left[\frac{d\mathbb{Q}_n}{d\bar{\mathbb{Q}}_n}(\bar \zb_n(t),\bar \xi_n(t))\chi_{\left\{\bar \zb_n\in N_{\epsilon/2}(\yb),\bar \nu_n\in N_{\epsilon/2}({\pi})\right\}}\right]\\ \notag
=&\mathbb{E}_{\bar{\mathbb{Q}}_{n}}\left[e^{B(\bar \zb_n,\bar \xi_n)}\chi_{\left\{\bar \zb_n\in N_{\epsilon/2}(\yb),\bar \nu_n\in N_{\epsilon/2}({\pi})\right\}}\right]\\ \label{Jensen}
\ge& \mathbb{E}_{\bar{\mathbb{Q}}_{n}}\big[\chi_{\left\{\bar \zb_n\in N_{\epsilon/2}(\yb),\bar \nu_n\in N_{\epsilon/2}({\pi})\right\}}\big]
\exp\left\{\frac{\mathbb{E}_{\bar{\mathbb{Q}}_{n}}\big[\chi_{\left\{\bar \zb_n\in N_{\epsilon/2}(\yb),\bar \nu_n\in N_{\epsilon/2}({\pi})\right\}}B(\bar \zb_n,\bar \xi_n)\big]}{\mathbb{E}_{\bar{\mathbb{Q}}_{n}}\big[\chi_{\left\{\bar \zb_n\in N_{\epsilon/2}(\yb),\bar \nu_n\in N_{\epsilon/2}({\pi})\right\}}\big]}\right\}.
\end{align}
By Lemmas \ref{covergence of tilde v} and \ref{covergence of tilde z}, we know that
\begin{equation}\label{ypi}
\mathop{\lim}\limits_{n\to \infty}  \mathbb{E}_{\bar{\mathbb{Q}}_{n}}\left[\chi_{\left\{\bar \zb_n\in N_{\epsilon/2}(\yb),\bar \nu_n\in N_{\epsilon/2}({\pi})\right\}}\right]=1.
\end{equation}
Thus, according to Lemma \ref{section 7 of Adam's book}, \eqref{Jensen} and \eqref{ypi}, we have
\begin{align}\notag
&\mathop{\lim \inf}\limits_{n\to \infty}\frac{1}{n}\log\mathbb{P}\left(\zb_n\in N_\epsilon(\rb),\nu_n\in N_\epsilon(\nu)\right)\\ \notag
\ge&\mathop{\lim \inf}\limits_{n\to \infty}\frac{1}{n}\mathbb{E}_{\bar{\mathbb{Q}}_{n}}\big[\chi_{\left\{\bar \zb_n\in N_{\epsilon/2}(\yb),\bar \nu_n\in N_{\epsilon/2}({\pi})\right\}}B(\bar \zb_n,\bar \xi_n)\big]\\ \notag
=&-\Bigg(\sum_{m=0}^{J-1}\int_{t_{m}}^{t_{m+1}}\sum_{i=1}^S\left(\lambda^{{\pi}}_{i}(\yb(t))-\mu_i^m\right)dt\\ \notag
+&\sum_{m=0}^{J-1}\int_{t_{m}}^{t_{m+1}}\sum_{i=1}^S\mu_i^m\log \frac{\mu_i^m}{\lambda^{{\pi}}_{i}(\yb(t))}dt \\  \label{dQn}
+&\sum_{m=0}^{J-1}\sum_{k=0}^{K_m-1}\int_{t_{mk}}^{t_{m,k+1}}\sum_{i=1}^D n_{\pi}(t,i)\sum_{j=1}^D\Big({\eta_{ij}^{mk}}\log \frac{{\eta_{ij}^{mk}}}{q_{ij}(\yb(t))}+q_{ij}(\yb(t))-{\eta_{ij}^{mk}}\Big)dt\Bigg).
\end{align}
Combining Lemma \ref{up 1}, Lemma \ref{up 2} and \eqref{dQn}, we finish
the proof.
\end{proof}
In the final theorem, we remove the continuity assumption on $n_\nu(t,\cdot)$ to get the desired lower bound estimation.
\begin{thm}\label{final theorem}
For given $\rb \in \mathbb{D}^d[0,T]$ and $\nu \in \mathbb{M}_L[0,T]$, assume that $\rb$ is absolutely continuous, we have
\begin{eqnarray*}
\mathop{\lim \inf}\limits_{n\to \infty}\frac{1}{n}\log\mathbb{P}\left(\zb_n\in N_\epsilon(\rb),\nu_n\in N_\epsilon(\nu)\right)\ge - \left(I_s(\rb,\nu)+I_{f}(\rb,\nu)\right).
\end{eqnarray*}
\end{thm}
\begin{proof}  We can construct a sequence of measures $\nu^{(k)}$ $(k \ge 1)$  such for any $k$, $n_{\nu^{(k)}}$ is continuous in $t$ and $\rho^{(2)}(\nu,\nu^{(k)}) \to 0$.  From Lemma \ref{lower semicontinuity of I+F}, $I_s(\rb,\nu)+I_{f}(\rb,\nu)$ is lower semi-continuous in $\nu$. Thus, we can choose $k_0$ large enough such that for any $\delta>0$ and $\epsilon>0$,
 $$ I_s(\rb,\nu^{(k_0)})+I_{f}(\rb,\nu^{(k_0))}) \ge I_s(\rb,\nu)+I_{f}(\rb,\nu)-\delta$$
 and
 $$\rho^{(2)}(\nu,\nu^{(k_0)})< \epsilon/2.$$
Thanks to Lemma \ref{main lemma}, we have
 \begin{eqnarray*}
&&\mathop{\lim \inf}\limits_{n\to \infty}\frac{1}{n}\log\mathbb{P}\left(\zb_n\in N_\epsilon(\rb),\nu_n\in N_\epsilon(\nu)\right) \\
&\ge& \mathop{\lim \inf }\limits_{n\to \infty}\frac{1}{n}\log\mathbb{P}\left(\zb_n\in N_{\epsilon}(\rb),\nu_n\in N_{\epsilon/2}(\nu^{(k_0)})\right) \\
&\ge&- \left(I_s(\rb,\nu^{(k_0)})+I_{f}(\rb,\nu^{(k_0)})\right)\\
&\ge&- \left(I_s(\rb,\nu)+I_{f}(\rb,\nu)\right)-\delta.
\end{eqnarray*}
The proof is completed.
\end{proof}

\subsection{Goodness of the rate functional}
The rate functional $I_{s}(\rb,\nu)+I_{f}(\rb,\nu)$ is  lower semicontinuous by  Lemma \ref{lower semicontinuity of I+F}. The goodness of the rate functional is a direct consequence of Lemma \ref{compactness of I+F}.

\section{Proof of Theorem \ref{second result}}

Now we prove Theorem \ref{second result} under the consideration $\rb\in W=\overline{(\mathbb{R}^{+})^{d}}$ instead of the whole space. The main clue of the proof is the same as the proof of Theorem \ref{Main result} except some technicalities to understand the behavior of jumps near the boundary of $W$. We will only focus on the key parts which is different from the proof of Theorem \ref{Main result}.

The difficulty in the proof of lower bound is that we can not use the change of measure formula directly, since some of the jump rates may diminish on the boundary. Mainly following  \cite{Shwartz2005}, We overcome this issue by  carefully analyzing the boundary behavior of the dynamics .

Let a $d$-dimensional unit vector $\vb := (1,1,\cdots,1)/\sqrt{d}$  and define the shifting ${\rb}_\delta(t)= \rb(t)+\delta \vb$ with $\delta>0$ a sufficiently small number. With similar approach in proving Lemma 5.1 in \cite{Shwartz2005}, we can show that
\begin{eqnarray}\label{t delta}
 \mathop{\lim \sup}_{\delta \to 0^{+}} \left(I_s({\rb}_\delta,\nu)+I_f({\rb}_\delta,\nu)\right)\le I_s(\rb,\nu)+I_f(\rb,\nu).
\end{eqnarray}

\noindent Next we will prove
\begin{eqnarray*}
\mathop{\lim \inf}\limits_{n\to \infty}\frac{1}{n}\log\mathbb{P}\left(\zb_n\in N_\delta(\rb),\nu_n\in N_\delta(\nu)\right)\ge - \left(I_s(\rb,\nu)+I_{f}(\rb,\nu)\right).
\end{eqnarray*}
 Denote by $V_a(\rb)$ the modulus of continuity of $\rb$ with size $a$, and set $\eta(a)=\max\{V_a(\rb),a\}$ so that $\eta^{-1}(a)\le a $.
 Now, fix $\delta$ and set $t_\delta=\eta^{-1}(\delta/3)$. Then, $t_\delta \le \delta/3$ and for $t\le t_\delta$,
$$
\mathop{\sup}_{0\le t \le t_\delta}\|\rb(0) + t\cdot \vb-\rb(t)\| \le t_\delta \cdot \|\vb\| +\eta(t_\delta)\le 2\delta/3.
$$
Therefore, for $0< \alpha <1/6$,
\begin{eqnarray*}
&&\mathbb{P}\left(\zb_n\in N_\delta(\rb),\nu_n\in N_\delta(\nu)\right)\ge \mathbb{P}\Big(\|{\zb}_n(t)-{\rb}(0)-t\cdot \vb\| \le \alpha \delta~\text{on}~t\in [0,t_\delta],\\
&& \hspace*{3cm}{\zb}_n \in N_{\delta}(\rb;[t_\delta,T]); \nu_n\in N_\delta(\nu)\Big),
\end{eqnarray*}
where $N_{\delta}(\rb;[t_\delta,T])$ is the $\delta$-neighborhood of $\rb$ restricted  on $t\in [t_\delta,T]$. Now, on this time interval
$$
\mathop{\sup}_{t_\delta \le t \le T} \|\rb(t)-{\rb}_{t_{\delta}}(t)\| \le \delta/3
$$
and, moreover, $d({\rb}_{t_{\delta}}(t),\partial G) \ge  t_\delta/\sqrt{d}$. Therefore, for any function $\ub$ on $t\in[t_{\delta},T]$, $\|\ub-{\rb}_{t_{\delta}}\| \le t_\delta/2\sqrt{d}$
implies that $\|\ub-\rb\|\le 5\delta/6$ and $d({\rb}_{t_{\delta}}(t),\partial G) \ge t_\delta/2\sqrt{d}$. Now define $A_\delta$ the $\alpha \delta$-neighborhood of $\rb_{0}+t_\delta \vb$, i.e.  $A_\delta$$:=B_{\alpha \delta}(\rb_{0}+t_\delta \vb)$ and let ${\rb}_{t_{\delta}}^{\yb}$ be the shift of ${\rb}_{t_{\delta}}$ such that ${\rb}_{t_{\delta}}^\yb(t_\delta)=\yb$. Then,
\begin{eqnarray*}
&& \mathbb{P}\left(\zb_n\in N_\delta(\rb),\nu_n\in N_\delta(\nu)\right)\ge\\
&& \hspace*{1cm}\mathbb{P}\Big(\|{\zb}_n(t)-\rb(0)-t\cdot \vb\| \le \alpha \delta~\text{on}~t\in[0,t_\delta];\nu_n\in N_{\delta}(\nu;[0,t_\delta])\Big)\\
&& \hspace*{1cm}\times \mathop{\inf}_{\yb\in A_\delta} \mathbb{P}_\yb\left({\zb}_n \in N_{\frac{t_\delta}{2\sqrt{d}}}({\rb}_{t_{\delta}}^\yb;[t_\delta,T]);\nu_n\in N_{\delta}(\nu;[t_\delta,T])\right).
\end{eqnarray*}

\noindent The first term satisfies a large deviation lower bound
\begin{eqnarray}\notag
&&\mathop{\lim \inf}_{n \to \infty} \frac{1}{n}\log \mathbb{P}\Big(\|{\zb}_n(t)-\rb(0)-t\cdot \vb\| \le \alpha \delta~\text{on}~t\in [0,t_\delta];\nu_n\in N_{\delta}(\nu;[0,t_\delta])\Big)\\\label{t delta and 0}
&&\ge -C t_\delta
\end{eqnarray}
by estimating the probability of a specific path $\zb_{n}$ lying in the $\alpha\delta$-neighbor-hood of the curve $\rb(0)+t\vb$. Because the paths in $N_{\frac{t_\delta}{2\sqrt{d}}}({\rb}_{t_{\delta}}^\yb;[t_\delta,T])$ are bounded away from the boundary uniformly for $\yb\in A_\delta$,  by Theorem \ref{final theorem}, we have
\begin{eqnarray}\notag
&&\mathop{\lim \inf}_{n \to \infty} \frac{1}{n}\log \mathop{\inf}_{\yb\in A_\delta} \mathbb{P}_\yb\left({\zb}_n \in N_{\frac{t_\delta}{2\sqrt{d}}}({\rb}_{t_{\delta}}^\yb;[t_\delta,T]);\nu_n\in N_{\epsilon}(\nu;[t_\delta,T])\right)\\ \notag
&\ge & -\Big(I^{[t_\delta,T]}_s({\rb}_{t_{\delta}},\nu)+I^{[t_\delta,T]}_f({\rb}_{t_{\delta}},\nu)\Big)\\ \label{t delta and t}
&\ge &-\Big(I_s({\rb}_{t_{\delta}},\nu)+I_f({\rb}_{t_{\delta}},\nu)\Big),
\end{eqnarray}
where $I^{[t_\delta,T]}_s({\rb}_{t_{\delta}},\nu)$ and $I^{[t_\delta,T]}_f({\rb}_{t_{\delta}},\nu)$ are rate functionals defined on the integration interval $[t_\delta,T]$. According to (\ref{t delta}), (\ref{t delta and 0}) and (\ref{t delta and t}), we proved the lower bound.

Next let us consider the upper bound.  At first  we note that since the rates $\lambda_{i}(\zb,j)$ satisfies the linear growth condition
$$\lambda_{i}(\zb,j)\le C(1+\|\zb\|),$$
it is easy to show that
$$\mathop{\lim}_{K\to \infty}\mathop{\lim \sup}_{n\to \infty}\frac{1}{n}\log \mathbb{P}(\mathop{\sup}_{0\le t \le T} \|\zb_n(t)\| >K)=-\infty$$
by simple moment estimates and Doob's martingale inequality. Consequently, it suffices to prove the large deviation estimates for bounded sets and we can assume $\lambda_{i}(\zb,j)$ are bounded.

We only need to recheck Lemma \ref{phi and  phi^delta} and Lemma \ref{step function}, since the other lemmas in upper bound estimates can be verified easily under the assumption that $\lambda_{i} (\zb,j)$ are bounded.  Thanks to Corollary 4.2 and Lemma 4.6 in \cite{Shwartz2005}, we can obtain that Lemma \ref{phi and  phi^delta} and Lemma \ref{step function} are also correct under Assumption \ref{ass:second}. Thus the upper bound is also established.

The goodness of the rate functional trivially holds under Assumption \ref{ass:second}. So we complete the proof of Theorem \ref{second result}.

\section*{Appendix}
\renewcommand{\thesection}{A}

\begin{lem}\label{semicontinuous}
Let $\{f_\alpha \}$ be a collection of lower semi-continuous functions on a metric space. Then the function $f$ define by $f(x)=\sup_\alpha f_\alpha(x)$ is lower semicontinuous.
\end{lem}
\begin{lem}\label{convex}
Let $\{f_\alpha \}$ be a collection of convex functions on a metric space. Then the function $f$ define by
$f(x)=\sup_\alpha f_\alpha(x)$ is convex.
\end{lem}
\begin{lem}\label{Min-Max}
Let $K(x,y)$ be a real-valued function, continuous in $(x,y)$ on $\mathbb{R}^d\times \mathbb{R}^D$, convex in $x$ for each $y$, and concave in $y$ for each $x$. Let two non-empty closed convex sets $U$ and $V$ be given, at least one of which is bounded. Then
$$ \mathop{\inf}_{x\in U}\mathop{\sup}_{y\in V}K(x,y)=\mathop{\sup}_{y\in V}\mathop{\inf}_{x\in U}K(x,y).$$
\end{lem}
The proof of Lemma \ref{Min-Max} may be referred to Corollary 37.3.2 of \cite{Rockafellar1970}.

\subsection{Part 1} Proof of lemmas  related to the upper bound estimate.

\begin{lem} \label{2d basis}
Let $\zb(t)\in \mathbb{R}^d$ be any measurable process for $t\in [0,T]$. Suppose there exist numbers $a$ and $\delta$ such that for each $\pb \in \mathbb{R}^d$ with $\|\pb\|=1$,
$$\mathbb{P}\left(\sup_{0\le t\le T}\langle \zb(t), \pb \rangle \ge a \right) \le \delta.$$
Then
$$\mathbb{P}\left(\sup_{0\le t\le T}\|\zb(t)\|\ge a\sqrt{d}\right) \le 2d \delta.$$
\end{lem}
\begin{proof}
It is not difficult to find that
$$\Big\{\sup_{0\le t\le T}\|\zb(t)\| \ge a\sqrt{d}\Big\} \subset \bigcup_{i=1}^{2d}\Big\{\sup_{0\le t\le T}\langle \zb(t), \pb_i \rangle \ge a\Big\}.$$
where $\pb_i:=\eb_i$, $\pb_{i+d}:=-\eb_i$ for $i=1,\ldots, d$, and $\eb_i$ are chosen as the canonical orthonormal basis in Euclidean space $\mathbb{R}^d$.
\end{proof}

In later texts, we will take an abused notation $\xi_{n}(t)=\eb_{i}\in \mathbb{R}^{D}$ when $\xi_n(t)=i\in \mathbb{Z}_{D}$. This will not bring confusion since $\xi_{n}(t)$ is considered as a multidimensional vector only when we take inner product with other vectors.
\begin{lem}\label{lemma1}
There exists a function $K:\mathbb{R}^{+}\rightarrow \mathbb{R}^{+}$ with
$$\mathop{\lim}\limits_{a\to \infty}K(a)/a=+\infty,$$
such that
\begin{equation}\label{eq:DistInit}
\mathbb{P}\left(\mathop{\sup}\limits_{0\le t\le T}\|\zb_n(t)-\zb_n(0)\|\ge a\right)\le 2d\exp\left(-nTK\left(\frac{a}{T}\right)\right).
\end{equation}
\end{lem}

\begin{proof} The inequality \eqref{eq:DistInit} holds trivially whenever $K(a/T)=0$. It suffice to prove the lemma when $a$ is large. For $\pb \in \mathbb{R}^d$, $\sigmab \in \mathbb{R}^D$ and any $\rho>0$, with the form of infinitesimal generator $\mathcal{L}_{n}$ (\ref{eq:fullsystem}), we define a mean one exponential martingale
\begin{align*}
M_t^{\sigmab}=&\exp \left(  {\left\langle {{\zb_n}\left( t \right) - {\zb_n}(0),\rho \pb } \right\rangle  - n\int_0^t {\sum\limits_{i = 1}^S {{\lambda _i}({\zb_n}(s),{\xi _n}(s))({e^{\left\langle {\rho \pb ,\ub_i/n} \right\rangle }} - 1)ds} } }  \right. \notag\\
+&\left.\langle {\xi _n}(t) - {\xi _n}(0),\sigmab \rangle  - n\int_0^t \sum\limits_{i = 1}^D \chi_{\{\xi _n(s) = i\}}\sum\limits_{j = 1}^D q_{ij} (\zb_n(s))(e^{\langle \sigmab ,\eb_{ij} \rangle } - 1)ds   \right).
\end{align*}

Define  $U = \max_{1\le i\le S}\|\ub_i\|$. Fix $\|\pb\|=1$, we have $$
 n{\int_0^t {\sum\limits_{i = 1}^S {{\lambda _i}({\zb_n}(s),{\xi _n}(s))({e^{\left\langle {\rho \pb ,\ub_i/n} \right\rangle }} - 1)ds} } } \le nt S\Lambda e^{U \rho/n}=: R(t,\rho)$$
by Assumption \ref{ass:main}. Hence we obtain
\begin{align*}
&\mathbb{P}\left(\mathop{\sup}\limits_{0\le t\le T}\left\langle {{\zb_n}\left( t \right) - {\zb_n}(0),\pb } \right\rangle\ge a\right)\\
=& \mathbb{P}\left(\mathop{\sup}\limits_{0\le t\le T}\exp \left(\rho\left\langle {{\zb_n}\left( t \right) - {\zb_n}(0),\pb } \right\rangle\right) \ge \exp(\rho a)\right)\\
\le& \mathbb{P}\left(\mathop{\sup}\limits_{0\le t\le T}M_t^{\sigmab=0}\ge \exp\Big(\rho a-R(T,\rho)\Big)\right)\\
\le& \exp\left( {nT\left[ {S\Lambda {e^{U\rho /n}} - \frac{\rho }{n}\frac{a}{T}} \right]} \right),
\end{align*}
where the inequality follows from Doob's martingale inequality.
Take$$
\rho=\frac{n}{U}\log \frac{a}{TS\Lambda U}>0.$$
Then it is not difficult to show that if we set
$$\tilde{K}(a)=\frac{a}{U}\left(\log \frac{a}{S\Lambda U}-1\right)$$
for $a$ large and $K(a)= 0$ otherwise, then $$
\mathbb{P}\left(\mathop{\sup}\limits_{0\le t\le T}\left\langle {{\zb_n}\left( t \right) - {\zb_n}(0),\pb } \right\rangle \ge a\right)\le \exp\left(-nT\tilde{K}\left(\frac{a}{T}\right)\right).$$
Define $K(a)=\tilde{K}(a/\sqrt{d})$, we get the desired estimate by applying Lemma \ref{2d basis}.
\end{proof}

\begin{cor}\label{cor of lemma1}
 There exist positive constants $c_1$ and $c_2$ independent of $t$ and $\tau$, such that for any $t,\tau\in[0,T]$ with $0 \le t+\tau \le T$,$$
\mathbb{P}\Big(\mathop {\sup}\limits_{t \le s \le t+\tau}\|\zb_n(s)-\zb_n(t)\| \ge a\Big) \le 2d\exp\left(-nac_1 \log\Big(\frac {ac_2}{\tau}\Big)\right).$$
\end{cor}

\noindent{\bf Proof of  Lemma \ref{distance between z and y}}. Consider a typical interval $[t_j^n,t_{j+1}^n]$. Since $\zb_n(t)$ and $\tilde{\zb}_n(t)$ agree at the endpoints of this interval, it is obvious that
$$\|\tilde{\zb}_n(t_j^n)-\tilde{\zb}_n(t_{j+1}^n)\|>\frac{\delta}{2}\quad{\rm implies}\quad \|\zb_n(t_{j+1}^n)-\zb_n(t_j^n)\|>\frac{\delta}{2}.$$
On the other hand, we have
$$\|\zb_n(t)-\zb_n(t_j^n)\|\ge\|\zb_n(t)-\tilde{\zb}_n(t)\|-\|\tilde{\zb}_n(t_{j+1}^n)-\tilde{\zb}_n(t_j^n)\|$$
since $\tilde{\zb}_n$ is piecewise linear and $\tilde{\zb}_n(t_j^n)=\zb_n(t_j^n)$. Therefore if $\|\zb_n(t)-\tilde{\zb}_n(t)\|>\delta$ for some $t$ in the $j$th interval, we must have$$
\mathop{\sup}\limits_{t_j^n\le t\le t_{j+1}^n}\|\zb_n(t)-\zb_n(t_j^n)\| \ge \delta/2.$$

Applying Corollary \ref{cor of lemma1} with $a=\delta/2$ and $\tau=T/n$ we obtain
$$\mathbb{P}\left(\mathop {\sup}\limits_{t_j^n \le t \le t_{j+1}^n}\|\zb_n(t)-\zb_n(t_j^n)\| \ge \delta/2\right) \le 2d \exp\left(-n\frac{\delta c_1}{2} \log\Big(\frac{n\delta c_3}{2}\Big)\right),$$
 where $c_3= c_2/T$. Thus,
\begin{align*}
\mathbb{P}(\rho^{(1)}(\zb_n,\tilde{\zb}_n)>\delta) \le& \sum_{j=0}^{n-1}\mathbb{P}\left(\mathop{\sup}\limits_{t_j^n \le t \le t_{j+1}^n}\|\zb_n(t)-\tilde{\zb}_n(t)\|>\delta\right)\\
\le& \sum_{j=0}^{n-1}\mathbb{P}\left(\mathop{\sup}\limits_{t_j^n \le t \le t_{j+1}^n}\|\zb_n(t)-\zb_n(t_j^n)\|>\delta/2\right)\\
\le& n\cdot 2d \exp\left(-n\frac{\delta c_1}{2} \log\Big(\frac{n\delta c_3}{2}\Big)\right).
\end{align*}
The result follows since $c_1$ and $c_3$ are positive constants.\qed

\begin{lem}\label{local lemma}
For any given bounded sets $A_1\in\mathbb{R}^d$ and $A_2\in\mathbb{R}^D$, we have that
\begin{eqnarray*}
&&\mathop{\lim \sup}\limits_{n \to \infty}{\mathbb{E}_{\xb,m}}\exp \left\{ n\left\langle {{\tilde{\zb}_n}\left( {\frac{T}
{n}} \right) - {\tilde{\zb}_n}(0),\pb } \right\rangle -{n\int_0^{T/n} {H_s^\delta  \left( {\xb,\pb ,n_{\nu_n}(t,\cdot)} \right)dt} } \right.\\
&&~~~~   +\left.\left\langle {\xi _n}\left( {\frac{T}
{n}} \right) - {\xi _n}(0),\sigmab  \right\rangle  +  n\int_0^{T/n} S^\delta(\xb,n_{\nu_n}(t,\cdot),\sigmab)dt \right\}\le 1.
\end{eqnarray*}
holds uniformly in $\xb\in\mathbb{R}^d$, $m\in \{1,2,\cdots,D\}$, $\pb\in A_1$ and $\sigmab\in A_2$, where $\mathbb{E}_{\xb,m}$ means the expectation with respect to the paths of $(\zb_{n},\xi_n)$ starting from $(\xb,m)$ at $t=0$.
\end{lem}

\begin{proof} For any $\pb\in A_1$ and $\sigmab\in A_2$, define the mean one exponential martingale
\begin{align*}
{M_t} =& \exp \left( n\left[ {\left\langle {{\zb_n}\left( t \right) - {\zb_n}(0),\pb } \right\rangle  - \int_0^t {\sum\limits_{i = 1}^S {{\lambda _i}({\zb_n}(s),{\xi _n}(s))({e^{\left\langle {\pb ,{\ub_i}} \right\rangle }} - 1)ds} } } \right] \right.\\
&+ \left. \left\langle {{\xi _n}(t) - {\xi _n}(0),\sigmab } \right\rangle  - n\int_0^t {\sum\limits_{i = 1}^D {\chi_{\{\xi _n(s) = i\}}\sum\limits_{j = 1}^D {{q_{ij}}} ({\zb_n}(s))({e^{\left\langle {\sigmab ,{\eb_{ij}}} \right\rangle }} - 1)ds} }  \right).
\end{align*}
Since $\tilde{\zb}_n(t_j^n)=\zb_n(t_j^n)$, for any $\pb\in A_1$ we have
\begin{align*}
1=&{\mathbb{E}_{\xb,m}}\exp \left\{ n\left[ \left\langle {{\tilde{\zb}_n}\left( {\frac{T}
{n}} \right) - {\tilde{\zb}_n}(0),\pb } \right\rangle \right.\right.\\
& \left.\left.- \int_0^{T/n} \sum\limits_{i = 1}^S \sum\limits_{j = 1}^D {\lambda _i}({\zb_n}(s),j)({e^{\left\langle {\pb ,{\ub_i}} \right\rangle }} - 1)\nu_n(ds,j)   \right] \right.\\
&+\left. \left\langle {{\xi _n}\left( {\frac{T}
{n}} \right) - {\xi _n}(0),\sigmab } \right\rangle  - n\int_0^{T/n} {\sum\limits_{i,j = 1}^D {{{q_{ij}}} ({\zb_n}(s))({e^{\left\langle {\sigmab ,{\eb_{ij}}} \right\rangle }} - 1)\nu_n(ds,i)} }  \right\}.
\end{align*}
By definition, the term
$$\sum\limits_{i = 1}^S \sum\limits_{j = 1}^D {{\lambda _i}({\zb_n}(s),j)({e^{\left\langle {\pb ,{\ub_i}} \right\rangle }} - 1)\nu_n(ds,j)} $$
can be written as $H_s(\zb_n(s),\pb,n_{\nu_n}(s,\cdot))ds$ and $$
-{\sum\limits_{i = 1}^D {\sum\limits_{j = 1}^D {{q_{ij}}} ({\zb_n}(s))({e^{\left\langle {\sigmab ,{\eb_{ij}}} \right\rangle }} - 1)\nu_n(ds,i)} }$$
 can be written as $ S(\zb_n(s),n_{\nu_n}(s,\cdot),\sigmab)ds$.

Let$$S_{\delta}= \left\{\omega:\mathop{\sup}\limits_{0\le t \le T/n}\|\zb_n(t)-\xb\|<\frac{\delta}{2}\right\}, $$
we have
\begin{eqnarray}\notag
1 &\geqslant& {\mathbb{E}_{\xb,m}} \chi_{S_\delta }\exp \left\{\left( {n\left\langle {{\tilde{\zb}_n}\left( {\frac{T}
{n}} \right) - {\tilde{\zb}_n}(0),\pb } \right\rangle  - n\int_0^{T/n} {H_s^\delta  \left( {\xb,\pb ,n_{\nu_n}(t,\cdot)} \right)dt} } \right) \right.\\\notag
&&+\left. \left\langle {{\xi _n}\left( {\frac{T}
{n}} \right) - {\xi _n}(0),\sigmab } \right\rangle  + n\int_0^{T/n} S^\delta(\xb,n_{\nu_n}(t,\cdot),\sigmab)dt  \right\}\\\notag
&=& {\mathbb{E}_{\xb,m}}\exp \left\{\left( {n\left\langle {{\tilde{\zb}_n}\left( {\frac{T}
{n}} \right) - {\tilde{\zb}_n}(0),\pb } \right\rangle  - n\int_0^{T/n} {H_s^\delta  \left( {\xb,\pb ,n_{\nu_n}(t,\cdot)} \right)dt} } \right) \right.\\\notag
&& +\left. \left\langle {{\xi _n}\left( {\frac{T}
{n}} \right) - {\xi _n}(0),\sigmab } \right\rangle  + n\int_0^{T/n} S^\delta(\xb,n_{\nu_n}(t,\cdot),\sigmab)dt  \right\}\\\notag
&&- {\mathbb{E}_{\xb,m}}\chi_{S_\delta^c } \exp  \left\{\left( {n\left\langle {{\tilde{\zb}_n}\left( {\frac{T}
{n}} \right) - {\tilde{\zb}_n}(0),\pb } \right\rangle  - n\int_0^{T/n} {H_s^\delta  \left( {\xb,\pb ,n_{\nu_n}(t,\cdot)} \right)dt} } \right) \right.\\ \label{kv}
&& +\left. \left\langle {{\xi _n}\left( {\frac{T}
{n}} \right) - {\xi _n}(0),\sigmab } \right\rangle  + n\int_0^{T/n} S^\delta(\xb,n_{\nu_n}(t,\cdot),\sigmab)dt  \right\}
\end{eqnarray}
Since $A_1$ and $A_2$ are bounded sets, there exist $B_1$ and $B_2$ such that $\|\pb\|\le B_1$ and $\|\sigmab\| \le B_2$.
From the Assumption \ref{ass:main} and the boundedness of $\pb$ and $\sigmab$,  we have
\begin{align}\notag
&{\mathbb{E}_{\xb,m}}\chi_{S_\delta^c } \exp  \left\{\left( {n\left\langle {{\tilde{\zb}_n}\left( {\frac{T}
{n}} \right) - {\tilde{\zb}_n}(0),\pb } \right\rangle  - n\int_0^{T/n} {H_s^\delta  \left( {\xb,\pb ,n_{\nu_n}(t,\cdot)} \right)dt} } \right) \right.\\\notag
&+\left. \left\langle {{\xi _n}\left( {\frac{T}
{n}} \right) - {\xi _n}(0),\sigmab } \right\rangle  + n\int_0^{T/n} S^\delta(\zb_n(t),n_{\nu_n}(t,\cdot),\sigmab)dt  \right\}\\\notag
\le&{\mathbb{E}_{\xb,m}}\left( \chi_{S_\delta ^c}\exp  \left(n\left\langle {{\tilde{\zb}_n}\left( {\frac{T}
{n}} \right) - {\tilde{\zb}_n}(0),\pb } \right\rangle +3K \right)\right)\\\notag
\leqslant& \sum\limits_{k = 1}^\infty  {\exp \left( {n(k + 1)\frac{\delta }
{2}|\pb |+3K} \right)}  \times {\mathbb{P}}\left( {\frac{{k\delta }}{2}
\leqslant \mathop {\sup }\limits_{0 \leqslant t \leqslant T/n} |{\zb_n}(t) - \xb|
\leqslant \frac{{\left( {k + 1} \right)\delta }}
{2}} \right)\notag \\
 \leqslant& \sum\limits_{k = 1}^\infty  {2d\exp \left( {n\left( {(k + 1)\frac{\delta }
{2}{B_1} - \frac{{k\delta {c_1}}}
{2}\log \left( {\frac{{k\delta {c_2}n}}{{2T}}} \right)} \right)} \right)}\times e^{3K} \rightarrow 0  \label{e1}
\end{align}
as $n$ goes to infinity for all $\xb\in \mathbb{R}^d$ with $\|\pb\|\le B_1$ and $\|\sigmab\| \le B_2$, where $K$ is a uniform bound depending on the bounds of $S^\delta(\cdot,\cdot,\cdot)$ and $H^\delta_s(\cdot,\cdot,\cdot)$ in the whole space, $B_1$, $B_2$ and $T$. Combining
\eqref{e1} and \eqref{kv}, we complete the proof.
\end{proof}

\begin{cor} \label{cor of local lemma}
For any fixed step functions $\thetab(t)\in \mathbb{R}^d$ and $\alphab(t)\in \mathbb{R}^D$,  there exist constants $C>0$ and $n_0$ such that$$
\mathbb{E} \exp \{nJ_n(\tilde{\zb}_n,\thetab,\nu_n,\alphab)\} \le C$$
for all $n>n_0$, where $J_n$ is defined  in \eqref{eq:Jn}.
\end{cor}
\begin{proof}
By definition
\begin{align}\notag
\exp \{&nJ_n(\tilde{\zb}_n,\thetab,\nu_n,\alphab)\}\\\notag
= \exp&\left\{\sum\limits_{j = 0}^{n - 1}   \Bigg(  n\left\langle {\tilde \zb_n( {t_{j + 1}^n} ) - \tilde \zb_n( {t_j^n} ),\thetab ( {t_j^n} )} \right\rangle \right.\\ \notag
&- n\int_{t_j^n}^{t_j^{n + 1}} {H_{s}^\delta  \left( {\tilde \zb_n(t_j^n),\thetab (t_j^n),n_{\nu_n}(t,\cdot)} \right)dt} \\\notag
 &+ \left. n\int_{t_j^n}^{t_j^{n + 1}} S^\delta(\tilde \zb_n(t_j^n),n_{\nu_n}(t,\cdot),\alphab(t_j^n))dt  \Big)\right\}\\\notag
 \end{align}
 \begin{align}\notag
 = \exp&\left\{\sum_{j = 0}^{n - 1}   \left(  n\left\langle {\tilde \zb_n( {t_{j + 1}^n}) - \tilde \zb_n( {t_j^n}),\thetab ( {t_j^n} )} \right\rangle \right.\right.\\ \notag
  &- n\int_{t_j^n}^{t_j^{n + 1}} {H_{s}^\delta  \left( {\tilde \zb_n(t_j^n),\thetab (t_j^n),n_{\nu_n}(t,\cdot)} \right)dt} \\\label{estimation of J}
 &+ \left. \left\langle {{\xi _n}( {t_{j+1}^n}) - {\xi _n}(t_j^n),\alphab(t_j^n) } \right\rangle +n\int_{t_j^n}^{t_j^{n + 1}} S^\delta(\tilde \zb_n(t_j^n),n_{\nu_n}(t,\cdot),\alphab(t_j^n))dt  \Bigg)
\right. \notag \\
 &\left.-\sum_{j=0}^{n-1}\left\langle {{\xi _n}( {t_{j+1}^n}) - {\xi _n}(t_j^n),\alphab(t_j^n) } \right\rangle\right\}.
\end{align}
Now $\alphab$ is a step function, let us first consider $\alpha(t)=\alpha_0$ on the interval $t\in[0,\tau]$. We have
\begin{eqnarray*}
&&\sum_{j=0}^{n-1} \chi_{\{t_{j+1}^{n}\le \tau\}}\left\langle {{\xi _n}( {t_{j+1}^n}) - {\xi _n}(t_j^n),\alphab_0 } \right\rangle\\
&=&\left\langle {{\xi _n}\left( \frac{\left\lfloor {n\tau} \right\rfloor}{n}
 \right) - {\xi _n}(0),\alphab_0 } \right\rangle
\end{eqnarray*}
where $\left\lfloor {a} \right\rfloor$ is the largest integer smaller than $a$.
 Since $\xi _n$ and $\alphab$  are bounded in [0,T], $\left|\left\langle {{\xi _n}\left( \frac{\left\lfloor {n\tau} \right\rfloor}{n}
 \right) - {\xi _n}(0),\alphab_0 } \right\rangle\right|$ is uniformly bounded. Repeating this argument on the finite number of intervals on which $\alphab$ are constants, we have that
 $\left|\sum_{j=0}^{n-1}\left\langle {{\xi _n}( {t_{j+1}^n}) - {\xi _n}(t_j^n),\alphab(t_j^n) } \right\rangle\right|$ is bounded.  Thus by (\ref{estimation of J}), Lemma \ref{local lemma} and the Markov property of $(\zb_n,\xi_n)$,
 $$ \mathop{\lim \sup}\limits_{n\to \infty} \mathbb{E} \exp \{nJ_n(\tilde{\zb}_n,\thetab,\nu_n,\alphab)\} \le C$$
 where $C$ is a positive constant.
 \end{proof}

\begin{lem}\label{uniformly absolute continuity}
(Uniformly absolute continuity)  Given $\rb \in \mathbb D^d[0,T]$ and $\nu \in \mathbb{M}_L[0,T]$. Let $I_s(\rb,\nu)+I_{f}(\rb,\nu)\le K$ and fix some $\epsilon>0$. Then there is a $\delta>0$, independent of $\rb$, such that for any collection of non-overlapping intervals in $[0,T]$ with total length $\delta$
$$\Big\{[t_j,s_j], j=1,\cdots, J\Big\}\quad {\rm with} \quad \sum_{j=1}^J (s_j-t_j)=\delta,$$
we have
$$ \sum_{j=1}^J \|\rb(s_j)-\rb(t_j)\|<\epsilon.$$
We can also find a constant $B$ depending only on $\epsilon$ and $K$ so that
$$\int_0^T \chi_{\{\|\dot \rb(t)\|\ge B\}}dt\le \epsilon.$$
\end{lem}

\begin{proof}
For any collection of non-overlapping intervals $\big\{[t_j,s_j]\big\}_j$, define the function $k(t)$ to be equal to one if $t$ is in some interval $[t_j,s_j]$ and zero otherwise. Since  $I_s(\rb,\nu)+I_{f}(\rb,\nu)\le K$, $\rb$ is absolutely continuous and $I_s(\rb,\nu)\le K$. For any $a>0$,
\begin{eqnarray*}
\sum_{j=1}^J \|\rb(s_j)-\rb(t_j)\| &\le& \int_0^T\|\dot \rb(t)\|k(t)dt\\
&\le& \int_0^T a\cdot \chi_{\{\|\dot \rb(t)\|\le a\}} k(t)dt\\
&&+\int_0^T\frac{L_s(\rb(t),\dot \rb(t),n_\nu(t,\cdot))}{L_s(\rb(t),\dot \rb(t),n_\nu(t,\cdot))/|\dot \rb(t)|}\chi_{\{\|\dot \rb(t)\|> a\}} k(t)dt\\
&\le& a \cdot \delta + \frac{K}{f(a)}
\end{eqnarray*}
where
$$f(a):=\mathop{\inf}\limits_{\zb,\betab, \wb \in\Delta_D }\left\{\frac{L_s(\zb,\betab,\wb)}{\|\betab\|}:\|\betab\|\ge a\right\}.$$
Recalling the definition of $L_s(\zb,\betab,\wb)$ in \eqref{eq:LsDef}, we define $U:=\max_i\|\ub_i\|$.
For any $\wb \in\Delta_D$ if we take
$\pb=\betab \log \|\betab\|/ (U \|\betab\|)$ in \eqref{eq:LsDef}, we obtain
$$L_{s}(\zb,\betab,\wb)\ge \|\betab\|\log \|\betab\|/ U - \|\betab\|S \Lambda.$$
This means that $f(a) \to \infty$ as $a\to \infty$. The choice $a=1/\sqrt{\delta}$ and taking $\delta$ sufficiently small establishes the uniformly absolute continuity.

Now we turn to the second statement. Since
\begin{eqnarray*}
\int_0^T \chi_{\{\|\dot \rb(t)\|\ge B\}}dt &\le& \frac {1}{B}\int_0^T\|\dot \rb(t)\|\chi_{\{\|\dot \rb(t)\|\ge B\}}dt\\
&\le& \frac {1}{B}\int_0^T\frac{L_s(\rb(t),\dot \rb(t),n_\nu(t,\cdot))}{L_s(\rb(t),\dot \rb(t),n_\nu(t,\cdot))/\|\dot \rb(t)\|}\chi_{{\{\|\dot \rb(t)\|\ge B\}}}dt\\
&\le& \frac {1}{B}\frac{I_s(\rb,\nu)}{f(B)},
\end{eqnarray*}
we complete the proof by choosing a sufficiently large $B$.
\end{proof}


\begin{lem}\label{lower semicontinuity of I+F}
The rate functionals are lower semicontinuous, i.e., if $(\rb_n,\nu_n) \to (\rb,\nu)$ as $n\to \infty$, then
\begin{eqnarray}\label{lower 1}
\mathop{\lim \inf}_{n\to \infty}I_s(\rb_n,\nu_n)\ge I_s(\rb,\nu), \quad \mathop{\lim \inf}_{n\to \infty}I_{f}(\rb_n,\nu_n)\ge I_{f}(\rb,\nu),
\end{eqnarray}
\begin{eqnarray}\label{lower 0}
\mathop{\lim \inf}_{n\to \infty}I_s^\delta(\rb_n,\nu_n)\ge I^\delta_s(\rb,\nu), \quad
 \mathop{\lim \inf}_{n\to \infty}I_f^\delta(\rb_n,\nu_n)\ge I^\delta_f(\rb,\nu)
\end{eqnarray}
and
\begin{eqnarray}\label{lower 2}
\mathop{\lim \inf}_{n\to \infty}I_s^\delta(\rb_n,\nu_n,\thetab)\ge I_s^\delta(\rb,\nu,\thetab), \quad  \mathop{\lim \inf}_{n\to \infty} I_{f}(\rb_n,\nu_n,\alphab)\ge I_{f}^\delta(\rb,\nu,\alphab)
\end{eqnarray}
for any fix step functions $\thetab(t)\in R^d$ and $\alphab(t) \in R^D$.
\end{lem}
\begin{proof} We only need to consider the sequences of $\rb_n$ which are absolutely continuous since it will be trivial otherwise. Let $(\rb_n,\nu_n) \to (\rb,\nu)$ as $n\to \infty$. We may assume that $I_s(\rb_n,\nu_n)+I_{f}(\rb_n,\nu_n)$ is bounded, say by a constant $K$. By Lemma \ref{uniformly absolute continuity}, we know that $\rb$ is also absolutely continuous.

Since $\rb(t)$ is absolutely continuous in $[0,T]$, given $\delta$, we can partition the interval $[0,T]$ into $J$ intervals $0=t_1\le t_2\le \cdots\le t_{J+1}=T$ each of length $\Delta$ such that
$$ \mathop{\max}_{j} \mathop{\sup}_{t_j\le t\le t_{j+1}}\|\rb_n(t)-\rb_n(t_j)\|< \delta.  $$
Denote $F_n(t,i)=\nu_n([0,t],i)$ and $F(t,i)=\nu([0,t],i)$. Recalling the definition of $L_{s}^\delta(\zb,\betab,\wb)$ in \eqref{eq:LsDeltaDef}, we have that $L_s^\delta(\zb,\betab,\wb))$ is lower semicontinuous in $\delta$, $\zb$, $\betab$ and $\wb$ and convex in $\betab$ and $\wb$ by Lemmas \ref{semicontinuous} and \ref{convex}. Thus for any $\epsilon>0$ and small enough $\Delta$, we have
\begin{align}\notag
&\int_0^TL_s(\rb_n(t),\dot{\rb}_n(t),n_{\nu_n}(t,\cdot))dt \\ \notag
\ge& \sum_{j=1}^{J}\int_{t_j}^{t_{j+1}}L_s^\delta(\rb_n(t_j),\dot{\rb}_n(t),n_{\nu_n}(t,\cdot))dt\\ \notag
\ge& \sum_{j=1}^{J}\Delta \cdot L_s^\delta\left(\rb_n(t_j),\frac{\int_{t_j}^{t_{j+1}}\dot{\rb}_n(t)dt}{\Delta},\frac{\int_{t_j}^{t_{j+1}}n_{\nu_n}(t,\cdot)dt}{\Delta}\right)dt\\ \label{lowsemi and convex0}
=& \sum_{j=1}^{J}\Delta \cdot L_s^\delta\left(\rb_n(t_j),\frac{\rb_n(t_{j+1})-\rb_n(t_{j})}{\Delta},\frac{F_n(t_{j+1},\cdot)-F_n(t_{j},\cdot)}{\Delta}\right).
\end{align}

Define the functions $\rb_J$, $F_J$ as
$$\rb_J(t)=\rb(t_j),\quad F_J(t,\cdot)=F(t_j,\cdot)\quad \text{for}~ t_j\le t<t_{j+1} , j=1,\cdots,J $$
and let
$$\rb^J(t):=\rb_J(t+\Delta),\quad F^J(t,\cdot):=F_J(t+\Delta,\cdot)\quad\text{for}~ 0\le t<T-\Delta.$$

By (\ref{lowsemi and convex0}), we have
\begin{eqnarray*}
&&\mathop{\lim \inf}_{n\to \infty}\int_0^TL_s(\rb_n(t),\dot\rb_n(t), n_{\nu_n}(t,\cdot))dt\\
&\ge&\sum_{j=1}^{J}\int_{t_j}^{t_{j+1}}\mathop{\lim \inf}_{n\to \infty} L_s^\delta\left(\rb_n(t_j),\frac{{\rb}_{n}(t_{j+1})-{\rb}_{n}(t_{j})}{\Delta},\frac{F_{n}(t_{j+1},\cdot)-F_{n}(t_{j},\cdot)}{\Delta}\right)dt\\
&\ge&\sum_{j=1}^{J-1}\int_{t_j}^{t_{j+1}} L_s^\delta\left(\rb_J(t),\frac{\rb^J(t)-\rb_J(t)}{\Delta},\frac{F^J(t,\cdot)-F_J(t,\cdot)}{\Delta}\right)dt\\
&=&\int_{0}^{T-\Delta}  L_s^\delta\left(\rb_J(t),\frac{\rb^J(t)-\rb_J(t)}{\Delta},\frac{F^J(t,\cdot)-F_J(t,\cdot)}{\Delta}\right)dt.
\end{eqnarray*}
Now we use the nested partitions $J_k=2^k$, so that $\Delta_k=T/2^k$ and a corresponding sequence $\delta_k$ that converges to zero.
By Fatou's Lemma,
 \begin{eqnarray*}
&&\mathop{\lim \inf}_{k\to \infty}\int_{0}^{T-\Delta_k} L_s^{\delta_k}\left(\rb_{J_k}(t),\frac{\rb^{J_k}(t)-\rb_{J_k}(t,\cdot)}{\Delta},\frac{F^{J_k}(t,\cdot)-F_{J_k}(t,\cdot)}{\Delta}\right)dt\\
&\ge&\int_{0}^{T}\mathop{\lim \inf}_{k\to \infty}\chi_{\{t\le T-\Delta_k\}}  L_s^{\delta_k}\left(\rb_{J_k}(t),\frac{\rb^{J_k}(t)-\rb_{J_k}(t,\cdot)}{\Delta},\frac{F^{J_k}(t,\cdot)-F_{J_k}(t,\cdot)}{\Delta}\right)dt\\
&\ge&\int_0^TL_s(\rb(t),\dot{\rb}(t),n_\nu(t,\cdot))dt.
\end{eqnarray*}
So we established the lower semicontinuity of $I_s(\rb,\nu)$.

 The lower semicontinuity of $I_f(\rb,\nu)$ can be done similarly.  Recall the definition of $S^\delta(\zb,\wb)$ in \eqref{Sdelta}, we have that $S^\delta(\zb,\wb)$ is lower semicontinuous in $\delta$, $\zb$ and $\wb$ and convex in $\wb$ by Lemmas \ref{semicontinuous} and \ref{convex}. With exactly similar procedure as proving the lower semicontinuity of $I_s(\rb,\nu)$, we can establish
\begin{eqnarray*}
\mathop{\lim \inf}_{n\to \infty}\int_0^TS(\rb_n(t),n_{\nu_n}(t,\cdot))dt \ge \int_{0}^{T-\Delta} S^\delta\left(\rb_J(t),\frac{F^J(t,\cdot)-F_J(t,\cdot)}{\Delta}\right)dt
\end{eqnarray*}
for a fine enough partition. Again we consider the sequence of nested partition $J_k=2^k$ and $\Delta_k=T/2^k$.
By Fatou's Lemma and the lower semicontinuity of $S$,
 \begin{eqnarray*}
&&\mathop{\lim \inf}_{k\to \infty}\int_{0}^{T-\Delta_k} S^{\delta_k}\left(\rb_{J_k}(t),\frac{F^{J_k}(t,\cdot)-F_{J_k}(t,\cdot)}{\Delta}\right)dt\\
&\ge&\int_{0}^{T}\mathop{\lim \inf}_{k\to \infty}\chi_{\{t\le T-\Delta_k\}}  S^{\delta_k}\left(\rb_{J_k}(t),\frac{F^{J_k}(t,\cdot)-F_{J_k}(t,\cdot)}{\Delta}\right)dt\\
&\ge&\int_0^TS(\rb(t),n_\nu(t,\cdot))dt.
\end{eqnarray*}
Thus we obtain the lower semicontinuity of $I_{f}(\rb,\nu)$. The proof of \eqref{lower 0} and \eqref{lower 2} are simlar.
\end{proof}

 \begin{lem}\label{step function}
 Given $\rb \in \mathbb D^d[0,T]$, $\nu \in \mathbb{M}_L[0,T]$ and $\epsilon>0$, there exist step functions $\thetab(t) \in R^d$ and $\alphab(t) \in R^D$ such that
  \begin{eqnarray}\label{eq:step0}
I_{s}(\rb,\nu,\thetab)\ge I_{s}(\rb,\nu)-\epsilon,
\end{eqnarray}
 \begin{eqnarray}\label{eq:step1}
I_{s}^\delta(\rb,\nu,\thetab)\ge I_{s}^\delta(\rb,\nu)-\epsilon,
\end{eqnarray}
and
 \begin{eqnarray}\label{eq:step2}
I_f(\rb,\nu,\alphab)\ge I_f(\rb,\nu)-\epsilon,
\end{eqnarray}
\begin{eqnarray}\label{eq:step3}
I_f^\delta(\rb,\nu,\alphab)\ge I_f^\delta(\rb,\nu)-\epsilon.
\end{eqnarray}
 \end{lem}

The proof of \eqref{eq:step0} and \eqref{eq:step1} can be referred to Lemma 5.43 in \cite{Shwartz1995} and the proof of \eqref{eq:step2} and \eqref{eq:step3} is similar. We will outline the main procedure here.

\begin{proof}
If $\rb$ is not absolutely continuous, $I_f(\rb,\nu,\alphab)=\infty$ by definition, so nothing needs to be proved. Now let us consider the case that $\rb$ is absolutely continuous. For convenience, let $L_{s}(\zb,\betab,\wb,\pb):=\left\langle\pb,\betab\right\rangle-H_{s}(\zb,\pb,\wb)$.
Since by definition $L_{s}(\zb,\betab,\wb,\pb)\le L_{s}(\zb,\betab,\wb)$ for any $\pb$,  we have for $B$ large enough
\begin{eqnarray*}
&&\int_0^T\chi_{\{\|\dot \rb(t)\|\ge B\}}L_{s}(\rb(t),\dot \rb(t),n_\nu(t,\cdot),\thetab(t))dt\\
&\le& \int_0^T\chi_{\{\|\dot \rb(t)\|\ge B\}}L_{s}(\rb(t),\dot \rb(t),n_\nu(t,\cdot))dt\\
&\le& \epsilon/4
\end{eqnarray*}
by Lemma \ref{uniformly absolute continuity}. Choose ${\thetab}_1(t)= {\boldsymbol 0}$ whenever $\|\dot \rb(t)\|\ge B$ or $\dot \rb(t)$ is not in $\mathcal{C}$ as defined in \eqref{eq:ConeDef}. Let $R:= \sup_{0\le t\le T}\|\rb(t)\|$.
Since $\rb$ is continuous, $R$ is finite. Simply replacing $\lambda_i$ by $\sum_{j=1}^D \lambda_i(\zb,j)w_j$ in Lemma 5.23 of \cite{Shwartz1995}, we have for $B_1$ large enough,
$$\mathop{\sup}_{|\pb|\le B_1}L_{s}(\zb,\betab,\wb,\pb)\ge L_{s}(\zb,\betab,\wb)-\frac{\epsilon}{8T}$$
for all $\|\zb\|\le R$, $\|\betab\|\le B$ in $\mathcal{C}$ and $\wb$ in $\Delta_D$.
So for any $(\zb,\betab,\wb)$ in bounded set
$$A:=\big\{\|\zb\|\le R, \betab \in \mathcal{C}, \|\betab\|\le B, \wb\in \Delta_D\big\},$$
there exist a $\pb_{\zb \betab \wb}$ with  $\|\pb_{\zb \betab \wb}\| \le B_1$ such that
$$L_{s}(\zb,\betab,\wb,\pb_{\zb \betab \wb})\ge L_{s}(\zb,\betab,\wb)-\frac{\epsilon}{4T}.$$
On the bounded set
$$\big\{\|\zb\|\le R, \betab \in \mathcal{C}, \|\betab\|\le B, \wb\in \Delta_D, \|\pb\| \le B_1\big\},$$
the function $L_{s}(\zb,\betab,\wb,\pb)$ is uniformly continuous. What's more, by Lemmas 5.22 and 5.33 in \cite{Shwartz1995},
$L_{s}(\zb,\betab,\wb)$ is continuous in $A$. So given any $(\zb,\betab,\wb) \in A $, there exist a $\delta_{\zb\betab\wb}>0$ such that
$$L_{s}(\tilde \zb, \tilde \betab, \tilde \wb,\pb_{\zb \betab \wb})\ge L_{s}( \tilde \zb, \tilde \betab, \tilde \wb)-\frac{\epsilon}{2T}$$
holds for any $(\tilde \zb, \tilde \betab, \tilde  ) \in O_{\zb\betab\wb}\cap A$, where $O_{\zb\betab\wb}$ is the $\delta_{\zb\betab\wb}$-neighborhood of $(\zb,\betab,\wb)$.
By Heine-Borel theorem, we can choose finite number of $O_{\zb^i\betab^j \wb^k}$ to cover $A$. It means that
$$L_{s}(\zb,\betab,\wb,\pb_{\zb^i\betab^j \wb^k})\ge L_{s}(\zb,\betab,\wb)-\frac{\epsilon}{2T}$$
whenever $\|\zb-{\zb}^i\|+\|\betab-{\betab}^j\|+\|\wb-{\wb}^k\|\le \delta_{\zb^i\betab^j \wb^k}$.

Define the function ${\thetab}_1(t)={\pb}_{\zb^i\betab^j \wb^k}$ whenever $ \|\rb(t)-{\zb}^i\|+\|\dot \rb(t)-{\betab}^j\|+\|n_\nu(t,\cdot)-{\wb}^k\| \le \delta_{\zb^i\betab^j \wb^k}$ with some tie-breaking rule. The function $\thetab_1(t)$ takes finite number of values. It may not be constant on intervals of time. So we approximate $\thetab_1(t)$ by a step function. Choose $\eta$ small enough such that
$$ \int_0^T \chi_{\{t \in A \}}L_{s}(\rb(t),\dot \rb(t),n_\nu(t,\cdot))dt\le \epsilon/4$$
whenever the set A has measure less than $\eta$. Since $\thetab_1(t)$ is a simple function,  we can approximate it by a step function $\thetab$ and it agrees with $\thetab_1$ outside of a set of measure $\eta$ (c.f. \cite{Roy}).We finish the proof for \eqref{eq:step0}  by collecting all approximations above.

For the proof for \eqref{eq:step2}, we take advantage of Lemma 5.23 in \cite{Shwartz1995} again by replacing $\lambda_i$ with $\sum_{j=1}^D w_i q_{ij}$. We have for $C$ large enough
$$\mathop{\sup}_{\|\sigmab\|\le C}S(\zb,\wb,\sigmab)\ge S(\zb,\wb)-\frac{\epsilon}{4T}.$$
On the bounded set
$$\big\{\|\sigmab\| \le C, \|\zb\|\le R, \wb\in \Delta_D\big\},$$
the function $S(\zb,\wb,\sigmab)$ is uniformly continuous. With the similar strategy for $L_s$ we can find the desired step function $\alphab$. So we finish the proof for \eqref{eq:step2}. The proof for $I_s^\delta$ and $I_f^\delta$ are similar.
\end{proof}

\subsection{Part 2} Proof of lemmas  related to the lower bound estimate.

\noindent{\bf Proof of Lemma \ref{up 1}}.
Since $S(\zb,\wb)$ is bounded by $Q:= \sum_{i,j=1}^{D}\mathop{\sup}\limits_{\zb}q_{ij}(\zb)$, there exists $(\eta(s),\psi(s)) \in \mathcal S$ for any $s \in [t_m,t_{m+1}]$ such that
\begin{align*}
&\sum_{i=1}^D \psi_i(s)\sum_{j=1}^D\Big(\eta_{ij}(s)\log \frac{\eta_{ij}(s)}{q_{ij}(\yb(s))}+q_{ij}(\yb(s))-\eta_{ij}(s)\Big)\\
\le& S(\yb(s),n_\nu(s,\cdot))+\epsilon
\end{align*}
and  $$|\psi_i(s)- n_\nu(s,i)|< \epsilon/{(8DT)}$$
by Lemma 8.61 in \cite{Shwartz1995}. For each fixed $s \in [t_m,t_{m+1}]$, there exists $\delta_{s}>0$ such that
\begin{align*}
&\sum_{i=1}^D \psi_i(s)\sum_{j=1}^D\Big(\eta_{ij}(s)\log \frac{\eta_{ij}(s)}{q_{ij}(\yb(s))}+q_{ij}(\yb(s))-\eta_{ij}(s)\Big)\\
\le& S(\yb(t),n_\nu(t,\cdot))+2\epsilon
\end{align*}
and
$$|\psi_i(s)- n_\nu(t,i)|< \epsilon/{(4DT)} $$
hold for  any $t \in O_{s} =(s-\delta_{s},s+\delta_{s}) \cap [t_m,t_{m+1}]$.
By Heine-Borel theorem, we can choose finite number of $O_{s_k}$ in ${\{O_{s}\}}_{s \in [t_m,t_{m+1}]}$ to
cover $[t_m,t_{m+1}]$. It means that there exists  a further subdivision of  interval $[t_m,t_{m+1}]$ (i.e., $t_m=t_{m0}<t_{m1}<\cdots<t_{mK_m}=t_{m+1}$) and related $(\psi^m({s_k}),\eta^{m}(s_k))\in \mathcal S $ such that for all $t \in [t_{mk},t_{m,k+1}]$
\begin{eqnarray*}
&&\sum_{i=1}^D \psi^m_i(s_k)\sum_{j=1}^D\Big(\eta_{ij}^{m}(s_k)\log \frac{\eta_{ij}^{m}(s_k)}{q_{ij}(\yb(t))}+q_{ij}(\yb(t))-\eta_{ij}^{m}(s_k)\Big)\\
&\le&S(\yb(t),n_\nu(t,\cdot))+2\epsilon.
\end{eqnarray*}

 Since $\log q_{ij}(\zb)$ are bounded and Lipschitz continuous in $\zb$, we can establish that $S(\zb,\wb)$ is absolutely continuous in $\zb$, and this absolute continuity is uniform in $\wb \in \Delta_D$. To see this, we only need to show that  the function
$$f(\xb,\wb):=\mathop{\sup}\limits_{\sigmab \in \mathbb{R}^D}\left( -\sum_{i,j=1}^{D}w_ix_{ij}\left(e^{\left\langle\sigmab,\eb_{ij}\right\rangle}
-1\right)\right)$$
is absolutely continuous in $\xb=(x_{11},x_{12},\ldots,x_{DD})\in [1/\Lambda,\Lambda]^{D^2}$  (as defined in \eqref{eq:LambdaLUBound}), uniformly in $\wb \in \Delta_D$. For any $\xb, \xb+\Delta \xb \in [1/\Lambda,\Lambda]^{D^2}$ with $\|\Delta \xb\|\le 1/4\Lambda$, let $h=1/4\Lambda$, $r=\|\Delta \xb\|/(h+\|\Delta \xb\|)$
and define $\qb=\xb+\Delta \xb/r$. With this construction, we have $\qb\in [1/2\Lambda,M+1/2\Lambda]^{D^2}$, $f(\xb,\wb)$, $f(\xb+\Delta \xb,\wb)$, $f(\qb,\wb)\in [0, (\Lambda+1/2\Lambda)D^2]$ and  $\xb+\Delta \xb=(1-r)\xb+r\qb$. From the convexity of $f(\xb,\wb)$ in $\xb$, we have
$$f(\xb+\Delta \xb,\wb)\le (1-r)f(\xb,\wb)+r f(\qb,\wb)$$
and thus
\begin{align*}
f(\xb+\Delta \xb,\wb)- f(\xb,\wb)\le r (f(\qb,\wb)-f(\xb,\wb))
\le 4\Lambda\Big(\Lambda+\frac{1}{2\Lambda}\Big)D^2\|\Delta \xb\|.
\end{align*}
The absolute continuity in $\zb$ and uniformity in $\wb$ of $S(\zb,\wb)$ ensures that the estimate
$$S(\yb(t),n_\nu(t,\cdot))\le S(\rb(t),n_\nu(t,\cdot))+\epsilon$$
holds when $J$ is large enough.

To simplify the notation, we will rewrite $\eta^{m}(s_k)$ as $\eta^{mk}$ and $\psi^m(s_k)$ as $\psi^{mk}$. So for each $m$, we have
\begin{align*}
&\sum_{k=0}^{K_m-1}\int_{t_{mk}}^{t_{m,k+1}}\sum_{i=1}^D \psi^{mk}_i\sum_{j=1}^D\Big(\eta_{ij}^{mk}\log \frac{\eta_{ij}^{mk}}{q_{ij}(\yb(t))}+q_{ij}(\yb(t))-\eta_{ij}^{mk}\Big)dt\\
\le&\int_{t_m}^{t_{m+1}}S(\rb(t),n_\nu(t,\cdot))dt+3(t_{m+1}-t_m)\epsilon.
\end{align*}
The proof is completed. \qed\\

\noindent{\bf Proof of  Lemma \ref{up 2}.}
Define $$\tilde f(\mub,\lambda(\zb,\cdot),\wb):=\sum_{i=1}^S\Big(\sum_{j=1}^D\lambda_i(\zb,j)w_j
-\mu_{i}+\mu_{i}\log \frac{\mu_{i}}{\sum_{j=1}^D\lambda_i(\zb,j)w_j}\Big)$$
and
$$\tilde L_s(\zb,\betab,\wb) =\mathop{\inf}_{\mub \in \mathcal K_{\betab}}\tilde f(\mub,\lambda(\zb,\cdot),\wb).$$
Taking advantage of Theorem 5.26 of \cite{Shwartz1995}, we have
\begin{equation}\label{eq:LsTLs}
\tilde L_s(\zb,\betab,\wb)= L_s(\zb,\betab,\wb).
\end{equation}
We will show that for any $B_1$, $\tilde L_s(\zb,\betab,\wb)$ is continuous in $\zb$ and $\wb$, uniformly in $\betab$ in
 $$\mathcal{V}:=\{\betab\in \mathcal C, \|\betab\|\le B_1\},$$
 where $\mathcal{C}$ is the cone defined in \eqref{eq:ConeDef}.

 By Lemma 5.20 of
\cite{Shwartz1995}, we can find a constant $B_2$ such that for any $\betab \in \mathcal V$ there exists a $\mub \in \mathcal K_{\betab}$ with $\|\mub\|\le B_2$. Therefore, for all $\betab \in \mathcal V$ and any
$\mub \in \mathcal K_{\betab}$ with $\|\mub\|\le B_2$,
\begin{align*}
&\tilde L_s(\zb',\betab,\wb')-\tilde L_s(\zb,\betab,\wb)\\
\le& \tilde f(\mub,\lambda(\zb',\cdot),\wb')-\tilde L_s(\zb,\betab,\wb)\\
\le& \tilde f(\mub,\lambda(\zb,\cdot),\wb')-\tilde L_s(\zb,\betab,\wb)+C_1\|\zb'-\zb\|\\
\le& \tilde f(\mub,\lambda(\zb,\cdot),\wb)-\tilde L_s(\zb,\betab,\wb)+C_1\|\zb'-\zb\|+C_2\|\wb'-\wb\|
\end{align*}
for some positive constants $C_1$ and $C_2$. Now choose $\mub$ to minimize $\tilde L_s(\zb,\betab,\wb)$ to establish that
\begin{align}\label{le B}
\tilde L_s(\zb',\betab,\wb')-\tilde L_s(\zb,\betab,\wb)\le C_1\|\zb'-\zb\|+C_2\|\wb'-\wb\|.
\end{align}
By Lemma 5.17 and Lemma 5.32 of \cite{Shwartz1995} (replacing $\lambda_i(\xb)$ with $\sum_{j=1}^D\lambda_i(\xb,j)w_j$), we know that there exist positive constants $M_1$, $M_2$ and $B$ so that for all $\betab\in \mathcal C$ with $\|\betab\|\ge B$ , all $\zb\in \mathbb R^d$ and all $\wb\in \Delta_D$,
\begin{align}\notag
     M_1 \|\betab\|\log  \|\betab\| \le \tilde L_s(\zb,\betab,\wb) \le  M_2 \|\betab\|\log  \|\betab\|.
\end{align}
So for any $\qb\in \mathbb D^d[0,T]$ and any $\tilde \nu \in \mathbb M[0,T]$,
\begin{align}\notag
&\int_0^T\chi_{\{\|\dot \rb(t)\|\ge B\}}\tilde L_s(\qb(t),\dot \rb(t), n_{\tilde \nu}(t,\cdot))dt\\\notag
\le& \int_0^T\chi_{\{\|\dot \rb(t)\|\ge B\}} M_2 \|\dot \rb(t)\|\log  \|\dot \rb(t)\|dt\\\notag
\le& \int_0^T \chi_{\{\|\dot \rb(t)\|\ge B\}}\frac{M_2}{M_1} \tilde L_s(\rb(t),\dot \rb(t), n_{\nu}(t,\cdot))dt\\\label{ge B}
:=&\epsilon(B)
\end{align}
By Lemma \ref{uniformly absolute continuity}, we have $\epsilon(B) \to 0$ as $B\to \infty$.
Combining \eqref{le B} and \eqref{ge B}, we have for any $\epsilon>0$, there exist a $\delta>0$ such that
$$\mathop{\sup}_{0\le t \le T}\|\qb(t)-\rb(t)\|< \delta \quad \text{and} \quad \mathop{\sup}_{0\le t \le T} \|n_{\tilde \nu}(t,\cdot)-n_\nu(t,\cdot)\|<\delta$$
implies
\begin{align}
\left|\int_0^T\tilde L_s(\rb(t),\dot \rb(t), n_\nu(t,\cdot))dt-\int_0^T\tilde L_s(\qb(t),\dot \rb(t), n_{\tilde \nu}(t,\cdot))dt\right|\le \epsilon.
\end{align}

With this continuity property, we have
\begin{align*}
&\int_0^T\tilde L_s(\rb(t),\dot \rb(t), n_\nu(t,\cdot))dt\\
=&\sum_{m=0}^{J-1}\int_{t_m}^{t_{m+1}} \tilde L_s(\rb(t),\dot \rb(t), n_\nu(t,\cdot))dt\\
\ge &\sum_{m=0}^{J-1}\int_{t_m}^{t_{m+1}} \tilde L_s(\rb(t_m),\dot \rb(t), n_{\pi}(t_m,\cdot))dt-\epsilon\\
\ge & \sum_{m=0}^{J-1}\Delta \cdot  \tilde L_s\left(\rb(t_m),\frac{\rb(t_{m+1})-\rb(t_m)}{\Delta}, n_{\pi}(t_m,\cdot)\right)-\epsilon.
\end{align*}
By definition of $\tilde L_s$, for each $m$, we have $\mub^m \in K_{\betab_m}$ such that
\begin{align*}
&\tilde L_s\left(\rb(t_m),\frac{\rb(t_{m+1})-\rb(t_m)}{\Delta}, n_{\pi}(t_m,\cdot)\right)\\
\ge &\sum_{i=1}^S\Big(\lambda^{\pi}_{i}(\yb(t_m))-\mu_i^m+\mu_i^m\log \frac{\mu_i^m}{\lambda^{\pi}_{i}(\yb(t_m))}\Big)-\epsilon/T
\end{align*}
and finally we have
\begin{align}
&\int_0^T\tilde L_s(\rb(t),\dot \rb(t), n_\nu(t,\cdot))dt\notag\\
\ge& \sum_{m=0}^{J-1}\Delta \cdot \sum_{i=1}^S\Big(\lambda^{\pi}_{i}(\yb(t_m))-\mu_i^m+\mu_i^m\log \frac{\mu_i^m}{\lambda^{\pi}_{i}(\yb(t_m))}\Big)-2\epsilon\notag\\
\ge&\int_0^T\sum_{i=1}^S\Big(\lambda^{\pi}_{i}(\yb(t))-\mu_i^m+\mu_i^m\log \frac{\mu_i^m}{\lambda^{\pi}_{i}(\yb(t))}\Big)dt-3\epsilon. \label{eq:TLsBound}
\end{align}
 Lemma \ref{up 2} is proved by combing \eqref{eq:LsTLs} and \eqref{eq:TLsBound}.
\qed

\vspace*{10pt}
\noindent{\bf Proof of  Lemma \ref{covergence of tilde v}}. We  need to prove that for any bounded continuous function $h(t,z)$,
$$ \mathop{\lim}_{n\to \infty}\int_0^Th(t,\bar \xi_n(t))dt=\int_0^T\sum_{i=1}^Dh(t,i)n_{\pi}(t,i)dt$$
in probability. It suffices to prove that for each time interval $[t_{mk},t_{m,k+1}]$,
$$ \mathop{\lim}_{n\to \infty}\int_{t_{mk}}^{t_{m,k+1}}h(t,\bar \xi_n(t))dt=\int_{t_{mk}}^{t_{m,k+1}}\sum_{i=1}^Dh(t,i)n_{\pi}(t,i)dt.$$
Since $\bar \xi_n$ lives on only finite states, then for any $\epsilon>0$, there exists $\delta>0$ such that for $|t_k-t|<\delta$\\
$$ |h(t,\bar \xi_n(t))-h(t_k,\bar \xi_n(t))|<\epsilon,$$
for all $t\in [t_{mk},t_{m,k+1}]$.

Take an integer $L$ large enough and define $\tilde\delta = (t_{m,k+1}-t_{mk})/L<\delta$. Let $\tau_l = t_{mk}+l\tilde\delta$ for $l=0,1,\ldots,L$.  We have
\begin{align}
 &\mathop{\lim \sup}_{n\to \infty}\int_{t_{mk}}^{t_{m,k+1}}h(t,\bar \xi_n(t))dt\nonumber\\
= & \mathop{\lim \sup}_{n\to \infty}\sum_{l=0}^{N-1}\int_{\tau_l}^{\tau_{l+1}}h(t,\bar \xi_n(t))dt\nonumber\\
\le&\mathop{\lim \sup}_{n\to \infty}\sum_{l=0}^{N-1}\int_{\tau_l}^{\tau_{l+1}}h(\tau_l,\bar \xi_n(t))dt+T \epsilon\nonumber\\
=&\sum_{k=0}^{N-1}\int_{\tau_l}^{\tau_{l+1}}\sum_{i=1}^Dh(\tau_l,i)n_{\pi}(t,i)dt+T\epsilon\label{eq:Ergodic}\\
\le&\sum_{k=0}^{N-1}\int_{\tau_l}^{\tau_{l+1}}\sum_{i=1}^Dh(t,i)n_{\pi}(t,i)dt+2T\epsilon\nonumber\\
=&\int_{t_{mk}}^{t_{m,k+1}}\sum_{i=1}^Dh(t,i)n_{\pi}(t,i)dt+2T\epsilon.\nonumber
\end{align}
In \eqref{eq:Ergodic} we utilized the ergodicity of the process $\bar\xi_{n}$ on each interval $[\tau_l,\tau_{l+1})$. The convergence can be obtained in the almost sure and $L_{\mathbb P}^1$-sense rather than in probability \cite{Book:Durrett}. Similarly, we can prove
$$\mathop{\lim \inf}_{n\to \infty}\int_0^Th(t,\bar \xi_n(t))dt \ge \int_{t_{mk}}^{t_{m,k+1}}\sum_{i=1}^Dh(t,i)n_{\pi}(t,i)dt-2T\epsilon.
$$
The proof is completed.\qed\\

\noindent{\bf Proof of Lemma \ref{covergence of tilde z}}.
The goal is to prove that for any $\epsilon>0$,
$$ \lim_{n \to \infty}\mathbb{P}\left(\mathop{\sup}\limits_{0\le t\le T}\|\bar \zb_n(t)-\yb(t)\| \ge\epsilon \right)=0.$$
For any $\pb \in \mathbb{R}^d$ and $\rho>0$, we have the martingale
\begin{align*}
{M_t} =\exp \Bigg\{ &\left\langle {{{\bar \zb}_n}\left( t \right) - {\yb }(t),\rho \pb } \right\rangle \\
 &- \int_0^t \sum\limits_{i = 1}^S \Big( n{\mu _i(s)}\frac{{{\lambda _i}({{\bar \zb}_n}(s),{\bar\xi _n}(s))}}{\lambda^{{\pi}}_{i}(\yb(s))}({e^{\left\langle {\rho \pb ,{\ub_i}/n} \right\rangle }} - 1) \\
&-{\mu _i(s)}\left\langle {\rho \pb ,{\ub_i}} \right\rangle  \Big) ds   \Bigg\}\\
=\exp \Bigg\{& \left\langle {{{\bar \zb}_n}\left( t \right) - {\yb }(t),\rho \pb } \right\rangle \\
 &- \int_0^t \sum\limits_{i = 1}^S \Big({\mu _i(s)}\frac{\lambda _i({{\bar \zb}_n}(s),{\bar\xi _n}(s))-\lambda^{{\pi}}_{i}(\yb(s))}{\lambda^{{\pi}}_{i}(\yb(s))}\left\langle {\rho \pb ,{\ub_i}} \right\rangle \\
&+{\mu _i(s)\frac{{{\lambda _i}({{\bar \zb}_n}(s),{\bar\xi _n}(s))}}
{\lambda^{{\pi}}_{i}(\yb(s))}}\big(n({e^{\left\langle {\rho \pb ,{\ub_i}/n} \right\rangle }} - 1)-\left\langle {\rho \pb ,{\ub_i}}
\right\rangle\big)\Big) ds   \Bigg\}.
\end{align*}
Recall the Assumption \ref{ass:main} and $\mu_i(t)$ is piecewise constant and bounded, we can perform similar estimate as in Lemma \ref{lemma1} to obtain
\begin{align}\notag
 &\mathbb{P}\Bigg(\mathop{\sup}\limits_{0\le t\le T}\bigg\|\bar \zb_n(t)-\yb(t)  - \int_0^t \sum\limits_{i = 1}^S {\mu _i(s)}\frac{\lambda _i({{\bar \zb}_n}(s),{\bar\xi _n}(s))-\lambda^{{\pi}}_{i}(\yb(s))}{\lambda^{{\pi}}_{i}(\yb(s))}ds\ub_i\bigg\| \ge\epsilon \Bigg )\\ \label{tilde z and y}
& \le \exp\Big( - n\epsilon c_1 \log(\epsilon c_2 )\Big),
\end{align}
where $c_1$ and $c_2$ are positive constants.
By Lemma \ref{covergence of tilde v}, we have
\begin{eqnarray}\notag
&&\int_0^t {{\lambda _i}({{\bar \zb}_n}(s),{\bar\xi _n}(s))ds}  - \int_0^t \lambda^{{\pi}}_{i}(\yb(s))ds  \\\notag
&=& \left( {\int_0^t {{\lambda _i}({{\bar \zb}_n}(s),{\bar\xi _n}(s))ds}  - \int_0^t {{\lambda _i}(\yb(s),{\bar\xi _n}(s))ds} } \right) \\\notag
&&+ \left( {\int_0^t {\sum^{D}_{j=1} {{\lambda _i}} (\yb(s),j){n_{{\bar\nu _n}}}(s,j)}ds  - \int_0^t \lambda^{{\pi}}_{i}(\yb(s)) ds } \right)\\ \label{numerator}
&\leqslant& K\int_0^t {\|{{\bar \zb}_n}(s) - \yb(s)\|}ds + B_n,
\end{eqnarray}
where
\begin{equation}\label{eq:Fn}
B_{n}=\sup_{t\in[0,T]}\left|\sum^{D}_{j=1} \int_0^t  {{\lambda _i}}(\yb(s),j)\Big(n_{\bar\nu _n}(s,j) - n_{\pi} (s,j)\Big)ds\right|\to 0
\end{equation}
 as $n$ goes to infinity for $t\le T$.

Define $C=dAUT\Lambda$, where
$$A= \mathop{\max}\limits_{t\in [0,T]}\mathop{\max}\limits_{i=1,\ldots,S} \mu_{i}(t)  \quad \text{and}\quad U=\mathop{\max}\limits_{i=1,\cdots, S} \|\ub_i\|.$$
Combining \eqref{tilde z and y}, \eqref{numerator} and \eqref{eq:LambdaLUBound}, we have
\begin{align}\notag
&\mathbb{P}\left(\mathop{\sup}\limits_{0\le t\le T}\Big(\left\|\bar \zb_n(t)-\yb(t)\right\|-CK\int_0^t {\left\|{{\bar \zb}_n}(s) - {\yb }(s)\right\|}ds -C B_{n}\Big)\ge \epsilon\right)\\\notag
\le& \mathbb{P}\left(\mathop{\sup}\limits_{0\le t\le T}\bigg\|\bar \zb_n(t)-\yb(t)  - \int_0^t \sum\limits_{i = 1}^S {\mu _i(t)}\frac{\lambda _i({{\bar \zb}_n}(s),{\bar\xi _n}(s))-\lambda^{{\pi}}_{i}(\yb(s))}{\lambda^{{\pi}}_{i}(\yb(s))}ds\ub_i\bigg\| \ge \epsilon \right )\\ \label{c1 and c2}
\le& \exp\Big( - n\epsilon c_1 \log(\epsilon c_2 )\Big).
\end{align}
From (\ref{c1 and c2}) and Gronwall's inequality, we obtain
\begin{eqnarray*}
&&\mathbb{P}\left( {\mathop {\sup }\limits_{0 \leqslant t \leqslant T} \|{{\bar \zb}_n}(t) - {\yb }(t)\| \geqslant (\epsilon+CB_{n}) {e^{CKT}}} \right)\\
&\leqslant& \mathbb{P}\left( {\mathop {\sup }\limits_{0 \leqslant t \leqslant T} \left( {\|{{\bar \zb}_n}(t) - {\yb }(t)\| - CK\int_0^t {\|{{\bar \zb}_n}(s) - {\yb }(s)\|} ds} \right) \geqslant \epsilon} \right)\\
&\leqslant& \exp\Big( - n\epsilon c_1 \log(\epsilon c_2 )\Big).
\end{eqnarray*}
Combing the condition \eqref{eq:Fn} and the inequality
\begin{eqnarray*}
&&\mathbb{P}\left( {\mathop {\sup }\limits_{0 \leqslant t \leqslant T} \|{{\bar \zb}_n}(t) - {\yb }(t)\| \geqslant 2\epsilon}e^{CKT} \right)\\
&\le&\mathbb{P}\left( {\mathop {\sup }\limits_{0 \leqslant t \leqslant T} \|{{\bar \zb}_n}(t) - {\yb }(t)\| \geqslant (\epsilon+CB_{n}) {e^{CKT}}} \right) +\mathbb{P}\left( B_{n} \geqslant \frac{\epsilon}{C} \right),
\end{eqnarray*}
we finish the proof.
\qed

\section*{Acknowledgement}

T. Li acknowledges the support of NSFC under grants 11171009, 11421101, 91530322 and the National Science Foundation for Excellent Young Scholars (Grant No. 11222114). They thank Weinan E, Yong Liu, and Xiaoguang Li for helpful discussions, and an anonymous referee for careful reading and valuable suggestions to improve the paper.

\bibliographystyle{unsrt}

\begin{thebibliography}{10}


\bibitem{Assaf2011}
M. Assaf, E. Roberts, and Z. Luthey-Schulten,
\newblock{\em Determining the Stability of Genetic Switches: Explicitly Accounting for mRNA Noise}, Phys. Rev. Lett. 106 (2011), 2048102.

\bibitem{BillingsleyBook}
P. Billingsley,
\newblock{\em Convergence of probability measures}, John Wiley and Sons, New York, 1999.

\bibitem{Chen1990}
M. Chen and Y. Lu,
\newblock {\em On evaluating the rate functions of large deviations for jump processes}, Acta. Math. Sin. 6 (1990), 206-219.

\bibitem{Chen2015}
Y. Chen, C. Lv, F. Li and T. Li,
\newblock {\em Distinguishing the rates of gene activation from phenotypic variations},
BMC Syst. Biol. 9(2015), 29.


\bibitem{Dembo1993}
 A. Dembo and  O. Zeitouni,
\newblock {\em Large deviations techniques and applications},
Springer-Verlag, New York, 2nd Edition, 1998.

\bibitem{Book:Durrett}
R. Durrett,
\newblock {\em Probability: Theory and Examples}, 4th edition, Cambridge University Press, Cambridge, 2010.

\bibitem{E2005}
W. E, D. Liu and E. Vanden-Eijnden,
\newblock{Analysis of multiscale methods for stochastic differential equations},  Comm. Pure App. Math. 58 (2005), 1544-1585.

\bibitem{E2007}
W. E, D. Liu and E. Vanden-Eijnden,
\newblock{Nested stochastic simulation algorithms for chemical kinetic systems with multiple time scales},  J. Comp. Phys. 221 (2007), 158-180.

\bibitem{Ellis1985}
R.S. Ellis,
\newblock{\em Entropy, large deviations and statistical mechanics}, Springer-Verlag, New York, 1985.

\bibitem{Elowitz2002}
M. Elowitz, A. Levine, E. Siggia and P. Swain,
\newblock{\em Stochastic gene expression in a single cell}, Science 297 (2002), 1183-1186.

\bibitem{Kurtz1986}
S.N. Ethier and T.G. Kurtz,
\newblock{\em Markov processes: characterization and convergence},  John Wiley and Sons, New York, 1986.

\bibitem{Freidlin1998}
M.I. Freidlin and A.D. Wentzel,
\newblock {\em Random perturbations of dynamical systems},
Springer-Verlag, New York, 2nd Edition, 1998.

\bibitem{Gillespie1977}
D.T. Gillespie,
\newblock{\em Exact stochastic simulation of coupled chemical reactions}, J. Phys. Chem.
81 (1977), 2340-2361.

\bibitem{Gillespie2000}
D.T. Gillespie,
\newblock{\em The chemical Langevin equation}, J. Chem. Phys.
113 (2000), 297-306.

\bibitem{Eric2008}
M. Heymann and E. Vanden-Eijnden,
\newblock {\em The geometric minimum action method: A least action principle on the space of curves},  Comm. Pure Appl. Math. 61 (2008), 1052-1117.

\bibitem{Kurtz2013}
H. Kang and T.G. Kurtz,
\newblock{\em Separation of time-scales and model reduction for stochastic reaction networks},
Ann. Appl. Probab. 23 (2013), 529-583.

\bibitem{Kurtz1972}
T.G. Kurtz,
\newblock{\em The relationship between stochastic and deterministic models for chemical reactions}, J. Chem. Phys. 57 (1972) , 2976-2978.

\bibitem{Kurtz1973}
T.G. Kurtz,
\newblock{\em A limit theorem for perturbed operator semigroups with applications to random evolutions}, J. Funct. Anal. 12 (1973), 55-67.

\bibitem{LiLin2016}
T. Li and F. Lin,
\newblock{\em Two-scale large deviations for chemical reaction kinetics through second quantization path integral}, J. Phys. A: Math. Theor. 49 (2016), 135204.

\bibitem{Liptser1996}
R. Liptser,
\newblock{\em Large deviations for two scaled diffusions }, Prob. Theory Relat. Fields 106 (1996), 71-104.

\bibitem{Lv2014}
C. Lv, X. Li, F. Li and T. Li,
\newblock{\em Constructing the energy landscape for genetic switching system driven by intrinsic noise}, PLoS ONE 9 (2014), e88167.

\bibitem{Lv2015}
C. Lv, X. Li, F. Li and T. Li,
\newblock{\em Energy landscape reveals that the budding yeast cell cycle is a robust and adaptive multi-stage process}, PLoS Comput. Biol.  11 (2015), e1004156.

\bibitem{Papa1977}
G. Papanicolaou,
\newblock{\em Introduction to the asymptotic analysis of stochastic equations}, Lectures in Mathematics, Vol. 16, American Mathematical Society, Rode Island, 1977.

\bibitem{Rockafellar1970}
R.T. Rockafellar,
\newblock{\em Convex Analysis}, Princeton University Press, Princeton,  1970.

\bibitem{Roy}
H.L. Royden,
\newblock{\em Real Analysis}, Macmillan, New York, 2nd Edition, 1968.

\bibitem{Shwartz1995}
A. Shwartz and A. Weiss,
\newblock{\em Large deivations for performance analysis: queues, communications and computing}, Chapman and Hall, London, 1995.

\bibitem{Shwartz2005}
A. Shwartz and A. Weiss,
\newblock{\em Large deviations with diminishing rates}, Math. Oper. Res. 30 (2005), 281-310.

\bibitem{VanKampen2007}
N.G. van Kampen,
\newblock{\em Stochastic Processes in Physics and Chemistry}, Elsevier, Amsterdam, 3rd Edition, 2007.

\bibitem{Veretennikov2000}
A.Yu. Veretennikov,
\newblock{\em On large deviations for SDEs with small diffusion and averaging}, Stoch. Process. Appl. 89 (2000), 69-79.

\bibitem{Veretennikov1999}
A.Yu. Veretennikov,
\newblock{\em On large deviations in the averaging principle for SDE's with a "full
dependence"}, Ann. Prob. 27 (1999), 284-296.

\bibitem{Wang2013}
K. Zhang, M. Sasai and J. Wang,
\newblock{\em Eddy current and coupled landscapes for nonadiabatic and nonequilibrium complex system dynamics}, Proc. Nat. Acad Sci. USA 110 (2013), 14930-14935.

\bibitem{Zhou2016}
P. Zhou and T. Li,
\newblock{\em Construction of the landscape for multi-stable systems: potential landscape, quasi-potenital, A-type integral and beyond}, J. Chem. Phys.  144 (2016), 94109.

\end{thebibliography}

\end{document}